\documentclass[final,5p,times,twocolumn]{elsarticle}

\usepackage{amsthm}
\usepackage{graphicx}

\usepackage{enumitem}
\usepackage{xcolor, color}

\usepackage[T1]{fontenc}

\usepackage{multirow}
\usepackage{nccmath}
\usepackage{amsfonts}
\usepackage{amsmath}
\usepackage{mathtools}
\usepackage{mathrsfs}
\usepackage{graphicx}
\usepackage{subcaption}
\usepackage{algpseudocode}
\usepackage{booktabs}
\usepackage{xparse}

\usepackage{bm}

\NewDocumentCommand{\basis}{om}{%
	\IfNoValueTF{#1}
	{\| \Phi^{#2}\|^2}
	{\langle \Phi^{#1},\phi^{#2} \rangle}%
}

\NewDocumentCommand{\pce}{om}{%
	\IfNoValueTF{#1}
	{\textsf{#2}}
	{\textsf{#2}^{#1}}%
}

\DeclareMathOperator{\tr}{tr}
	
\newcommand{\mbb}[1]{\mathbb{#1}}
 
\newcommand{\mcl}[1]{\mathcal{#1}}

\newcommand{\splk}[2]{\mcl{L}^2(\Omega, \mathcal{F}_{#1}, \mu; \mathbb{R}^{#2})}
\newcommand{\splx}[1]{\mcl{L}^2(\Omega, \mathcal{F}, \mu; \mathbb{R}^{#1})}
\newcommand{\lsp}{\mcl{L}^2} 

\newcommand{\diff}{\mathop{}\!\mathrm{d}}

\newcommand{\mean}{\mbb{E}}

\newcommand{\ini}{_{\text{ini}}}
\newcommand{\tini}{\text{ini}}

\newcommand{\dimx}{{n_x}}
\newcommand{\dimu}{{n_u}}

\newcommand{\dimz}{{n_z}}

\newcommand{\I}{\mathbb{I}}
\newcommand{\N}{\mathbb{N}}
\newcommand{\R}{\mathbb{R}}

\newcommand{\Xini}{X_{\text{ini}}}

\newcommand{\ed}{\stackrel{\mcl{W}}{=}}

\newtheorem{lem}{Lemma}
\newtheorem{thm}{Theorem}
\newtheorem{defn}{Definition}

\newtheorem{assum}{Assumption}
\newtheorem{prop}{Proposition}

\newcommand{\sinf}{\diamond}

\journal{Journal of \LaTeX\ Templates}
\biboptions{authoryear}

\begin{document}
\begin{frontmatter}
\title{A Polynomial Chaos Approach to Stochastic LQ Optimal Control: Error Bounds and Infinite-Horizon Results}

\author[Hamburg]{Ruchuan Ou}\ead{ruchuan.ou@tu-dortmund.de}
\author[Bayreuth]{Jonas Schie{\ss}l}\ead{jonas.schiessl@uni-bayreuth.de}
\author[Bayreuth]{Michael Heinrich Baumann}\ead{michael.baumann@uni-bayreuth.de}
\author[Bayreuth]{Lars Gr{\"u}ne}\ead{lars.gruene@uni-bayreuth.de}

\author[Hamburg]{Timm Faulwasser\corref{mycorrespondingauthor}}
\cortext[mycorrespondingauthor]{Corresponding author}
\ead{timm.faulwasser@tu-dortmund.de}

\address[Hamburg]{Institute of Energy Systems, Energy Efficiency and Energy Economics, TU Dortmund University, 44227 Dortmund, Germany, and\\
	Institute of Control Systems, Hamburg University of Technology, 21079 Hamburg, Germany}
\address[Bayreuth]{Mathematical Institute, University of Bayreuth, 95440 Bayreuth, Germany}                                             


\begin{abstract}
The stochastic linear--quadratic regulator problem subject to Gaussian disturbances is well known and usually addressed via a moment-based reformulation. Here, we leverage polynomial chaos expansions, which model random variables via series expansions in a suitable $\lsp$ probability space, to tackle the non-Gaussian case. We present the optimal solutions for finite and infinite horizons and we analyze the infinite-horizon asymptotics. We show that the limit of the optimal state-input trajectory is the unique solution to a corresponding stochastic stationary optimization problem in the sense of probability measures. Moreover, we provide a constructive error analysis for finite-dimensional polynomial chaos approximations of the optimal solutions and of the optimal stationary pair in non-Gaussian settings. A numerical example illustrates our findings.
\end{abstract}
\end{frontmatter}

\section{INTRODUCTION}
The Linear--Quadratic Regulator (LQR) \citep{kalman60contributions,kalman60new} is one of the seminal results of optimal control. For stochastic Linear Time Invariant (LTI) systems, most works derive the optimal controller subject to Gaussian uncertainties via a moment-based reformulation \citep{astrom70introduction,anderson79optimal}. There is also a line of research which generalizes the results in different directions, e.g., \citet{lim99stochastic} extend to indefinite control weights, \citet{gattami09generalized} considers arbitrary disturbances with power constraints, \cite{singh17extended} solve the LQR problem for discrete-time LTI systems with perfect prior knowledge of the future and present disturbance samples. \citet{sun18stochastic}  consider stochastic infinite-horizon LQ Optimal Control Problems (OCP) subject to continuous-time LTI systems.

In the context of stochastic uncertainty, Polynomial Chaos Expansions (PCE) are based on series expansions of random variables and date back to \citet{wiener38homogeneous}. The core idea of PCE is that square integrable random variables can be modeled as $\mcl{L}^2$ functions in a Hilbert space and thus can be parameterized by deterministic coefficients in appropriately chosen polynomial bases. We refer to \citet{sullivan15introduction} for a general introduction and to \citet{kim13wiener} for an early overview on control design using PCE. It has been popularized for numerical implementation in stochastic optimal control by \citet{fagiano12nonlinear,paulson14fast}, while it also has prospects for theoretical analysis \citep{paulson15stability,ahbe20region,pan23stochastic, faulwasser23behavioral}. The early work by \citet{fisher09linear}, where probabilistic uncertainties in system matrices and not exogenous stochastic disturbances are considered, used PCE for stochastic LQR design. In a similar setting, PCE has also been used for robust control~\citep{templeton12probabilistic,wan23polynomial} and stochastic optimal control~\citep{kim12probabilistic}. Another work by \citet{levajkovic18solving} solved stochastic LQR problems subject to continuous-time LTI systems with Gaussian disturbances in the PCE framework.

The main goal of this paper is to obtain the solutions to discrete-time stochastic Linear Quadratic (LQ) optimal control problems and to the corresponding stationary optimization problems subject to non-Gaussian uncertainties in the PCE framework. As the equivalence between the deterministic and stochastic LQR problems with Gaussian disturbances is well-known, this work generalizes it for arbitrary uncertainties of finite expectation and variance. Our contributions are as follows:
(i) We show that a PCE-reformulated stochastic LQ OCP can be decomposed into many separable (i.e. decoupled) subproblems, each of which corresponds to one source of stochastic uncertainty in the system, i.e., uncertain initial state and process disturbances at each time step are treated separately.
(ii) This way, we deviate from the established moment-based reformulation and we present the optimal solutions to the considered OCPs for both finite and infinite horizons. In particular, we provide constructive error analysis for truncated PCEs and we analyze the convergence of infinite-horizon optimal solutions. (iii) We characterize the corresponding stationary optimization problem and we describe its unique solution in closed form. As the solution is of infinite dimension, we propose a finite-dimensional approximation with detailed error analysis. Especially, for an arbitrary desired error bound, the proposed scheme can determine the required dimension of the PCE approximation. Drawing upon an example, we demonstrate the procedure for numerical computation of the solution to a stationary optimization problem.

The remainder of the paper is structured as follows: Section~\ref{sec:Preliminaries} details settings and preliminaries of the considered stochastic LQ OCP. In Section~\ref{sec:FiniteLQ}, we present the stochastic LQR for finite horizon. Then we extend the results to the infinite-horizon case and analyze the asymptotics in Section~\ref{sec:InfiniteLQ}. Section~\ref{sec:Stationary} addresses the stochastic stationary optimization problem. Section~\ref{sec:Simulation} presents a numerical example. The paper ends with conclusions in Section~\ref{sec:Conclusion}.
\section{Preliminaries} \label{sec:Preliminaries}
We consider stochastic discrete-time LTI systems
\begin{equation} \label{eq:Sys}
	X_{k+1} = AX_k + BU_k + EW_k, \quad X_{0} = X\ini,
\end{equation}
with state $X_k \in \splk{k}{n_x}$ and process disturbance  $W_k \in \splx{n_w}$, where $\Omega$ is the set of realizations, $\mathcal{F}$ is a $\sigma$-algebra, and $\mu$ is the considered probability measure. Throughout the paper, the probability distributions of the disturbance $W_k$, $k\in \N$ and the initial condition $X\ini\in \splk{0}{n_x}$ are assumed to be known and $W_k$, $k\in \N$ are \textit{i.i.d.} random variables. For the sake of simplicity, the spaces $\splx{n_z}$ and $\splk{k}{n_z}$ are compactly written as $\lsp(\R^\dimz)$, respectively, as $\lsp_{k}(\R^\dimz)$.

In the filtered probability space $(\Omega, \mcl F, (\mcl F_k)_{k\in \N}, \mu)$, the $\sigma$-algebra contains all available historical information, i.e.,
$\mcl F_0 \subseteq \mcl F_1 \subseteq ...  \subseteq \mcl F$.
Let $(\mcl F_k)_{k\in \N}$ be the smallest filtration that the stochastic process $X$ is adapted to, i.e., $ \mcl F_k = \sigma(X_i,i\leq k)$, where $\sigma(X_i,i\leq k)$ denotes the $\sigma$-algebra generated by $X_i,i\leq k$. Then, the control at time step $k$ is modeled as a stochastic process which is adapted to the filtration $\mathcal{F}_k$, i.e. $U_k \in \lsp_k(\R^\dimu)$. This immediately imposes a causality constraint on $U_k$, i.e., $U_k$ depends only on $X_i$, $i\leq k$ up to time step $k$. Thus $U_k$ may only depend on past disturbances $W_i$, $i< k$. For more details on filtrations we refer to \citet{fristedt13modern}.

\subsection{Problem Statement}
To formulate the cost functional, first we recall the weighted norm of a vector-valued random variable $Z\in\lsp(\R^\dimz)$ as
\begin{equation*}
\|Z\|_Q \coloneqq \sqrt{\int_{\Omega} Z(\omega)^\top QZ(\omega) \diff \mu(\omega)}=\sqrt{\mean [Z^\top Q Z]}
\end{equation*}
for a symmetric and positive semidefinite matrix $Q\in\R^{n_z\times n_z}$. When $Q$ is the identity $I$, the above definition turns out to be the $\lsp$-norm of $Z$ and is denoted by the shorthand $\|Z\|\coloneqq \|Z\|_I$. The definition readily includes deterministic variables $z\in \R^{n_z}$ considering the distribution to be the Dirac distribution.

Given the initial condition $\Xini$ and the disturbance $W_k$, $k\in\N$, we consider the following stochastic LQ problem
	\begin{equation} \label{eq:StochOCP}
			\min_{ U_k \in \lsp_k(\R^\dimu),\atop k\in\I_{[0,N-1]}}~  \|X_N\|_{Q_N}^2 + \sum_{k=0}^{N-1}  \ell(X_k,U_k) \quad \text{s.t.}~ \eqref{eq:Sys},\quad~
	\end{equation}
where $\ell(X_k,U_k) \coloneqq \|X_k\|_Q^2+\|U_k\|_R^2$, $Q_N\succeq 0$, $Q\succeq 0$, and $R\succ 0$. $\I_{[0,N-1]}$ denotes the set of integers $\{0,1,...,N-1\}$, $N \in\N$. The cost functional evaluated along an input sequence $\{U_k\}_{k=0}^{N-1}$ is written as $J_N(\Xini,U)$, while the minimum $J_N(\Xini,U^\star)$ is obtained for the optimal input $\{U_k^\star\}_{k=0}^{N-1}$. It directly follows from Lemma~1.14 by \citet{kallenberg97foundations} that inputs $U_k$, $k\in\N$ adapted to the filtration $\mcl F_k$ are equivalent to state feedback polices. Throughout the paper, we assume that $(A,B)$ is stabilizable and that $(A,Q^{1/2})$ is detectable.
 
\subsection{Polynomial Chaos Expansion}
PCE is a well-established framework for propagating uncertainties through dynamics. It was first introduced by \citet{wiener38homogeneous} to model stochastic processes using Hermite polynomials with Gaussian random variables. PCE was further generalized to other orthogonal polynomials for any $\lsp$ stochastic processes by~\citet{xiu02wiener}, while \citet{ernst12convergence} analyzed the convergence properties of the generalized PCEs. For a concise overview on PCE and its use in systems and control, we refer to \citet{kim13wiener}.

The core idea of PCE is that any $\lsp$ random variable can be described in a suitable polynomial basis. Consider an orthogonal polynomial basis $\{\phi^j(\xi)\}_{j=0}^{\infty}$ that spans the space $\mcl{L}^2(\Xi, \mathcal{F}, \mu; \mathbb{R})$, where the random variable $\xi\in \lsp(\R^{n_\xi})$ is called the stochastic germ of polynomials $\phi^j$, and $\Xi$ is the sample space of~$\xi$. Then it satisfies the following orthogonality relation
\begin{equation} \label{eq:Orthogonality}
	\langle \phi^i(\xi),\phi^j(\xi) \rangle {=} \int_{\Xi} \phi^i(\xi) \phi^j(\xi) \diff \mu(\xi) {=} \delta^{ij}\| \phi^j(\xi) \|^2,
\end{equation}
where $\delta^{ij}$ is the Kronecker delta and $\| \phi^j(\xi) \|^2= \langle \phi^j(\xi),\phi^j(\xi) \rangle$ by definition. The first polynomial is always chosen to be $\phi^0(\xi) = 1$. Hence, the orthogonality~\eqref{eq:Orthogonality} gives that for all other basis dimensions $j>0$, we have $\mean[\phi^j(\xi)]=\int_{\Xi} \phi^j(\xi) \diff \mu(\xi)=0$.

\begin{defn}[Polynomial chaos expansion]
The PCE of a real-valued random variable $Z \in  \lsp(\R)$ with respect to the basis $\{\phi^j(\xi)\}_{j=0}^{\infty}$ is 
\[
	Z(\omega) = \sum_{j=0}^{\infty}\pce{z}^j \phi^j(\xi(\omega)) \quad \text{with} \quad \pce{z}^j = \frac{\big\langle Z(\omega), \phi^j(\xi(\omega)) \big\rangle}{\| \phi^j(\xi) \|^2},
\]
where $\pce{z}^j\in \R$ is referred to as the $j$-th PCE coefficient.
\end{defn}
Compared to many other spectral representations of random variables and random processes, e.g. Karhunen–Lo{\` e}ve expansion consisting of coefficients in random variables and real-valued functions, we get \emph{deterministic} PCE coefficients $ \pce{z}^j$ and thus can treat the random variable $Z$ deterministically in the PCE framework \citep{ghanem91stochastic}. The stochastic germ $\xi: \Omega\to\Xi$ is the random variable argument of the polynomial basis. That is, $\xi(\omega)$ is viewed as a function of the outcome~$\omega$. This way, we construct the mapping between the random variable $Z$ and the stochastic germ $\xi$ in the PCE representation. The PCE of a vector-valued random variable $Z\in\lsp(\R^\dimz)$ follows by applying PCE component-wise, i.e., the $j$-th PCE coefficient of $Z$ reads
$\pce{z}^{j} \coloneqq \begin{bmatrix} \pce{z}^{1,j} & \pce{z}^{2,j} & \cdots & \pce{z}^{n_z,j} \end{bmatrix}^\top$, where $\pce{z}^{i,j}$ is the $j$-th PCE coefficient of $i$-th component $Z^i$.

By replacing all random variables in \eqref{eq:Sys} with their PCEs and using one joint basis $\{\phi^j(\xi)\}_{j=0}^\infty$, we obtain
\[
	\textstyle{\sum_{j=0}^{\infty}} \pce{x}_{k+1}^j\phi^j(\xi) = \textstyle{\sum_{j=0}^{\infty}} \Big(A\pce{x}_k^j + B\pce{u}_k^j+ E\pce{w}_k^j \Big)\phi^j(\xi).
\]
Projecting the above equation onto $\phi^j(\xi)$, the orthogonality relation~\eqref{eq:Orthogonality} indicates that for all $j \in \N^\infty$, given $\pce{x}^j\ini$ and $\pce{w}_k^j$, $k \in \N$, the PCE coefficients satisfy
\begin{equation} \label{eq:SysPCE}
	\pce{x}_{k+1}^j = A\pce{x}_k^j +B\pce{u}_k^j+ E\pce{w}_k^j,\quad \pce[j]{x}_{0} = \pce[j]{x}\ini
\end{equation}
for all $j \in \N^\infty$ with $ \N^\infty\coloneqq \N\cup\{\infty\}$. This procedure is known as Galerkin projection and we refer to \citet{pan23stochastic}, Appendix A for details and further references.

The truncation error $\Delta Z(L) = Z - \sum_{j=0}^{L-1}\pce{z}^j \phi^j(\xi)$, where the argument~$L\in\N^\infty$ is the  PCE dimension, satisfies $\lim_{L\to\infty}\|\Delta Z(L)||=0$ \citep{cameron47orthogonal, ernst12convergence}. Moreover, \citet{xiu02wiener} show that in appropriately chosen polynomial bases, the convergence rate to the limit is exponential in the $\lsp$ sense.
\begin{defn}[Exact PCE representation] \label{def:ExactPCE}
	We say a random variable $Z \in \lsp(\R^\dimz)$ admits an exact PCE of finite dimension $L\in\N$ if $ Z - \sum_{j=0}^{L-1}\pce[j]{z} \phi^j(\xi)=0$.
\end{defn}
Moreover, consider the PCEs $Z=\sum_{j=0}^{L-1}\pce{z}^j\phi^j(\xi)$ and $\tilde{Z}=\sum_{j=0}^{L-1}\tilde{\pce{z}}^j\phi^j(\xi)$ in the same basis $\{\phi^j(\xi)\}_{j=0}^{L-1}$, the expectation $\mean[Z]$ and the covariance $\Sigma[Z,\tilde{Z}]$ can be calculated as \citep{lefebvre20moment}
\begin{equation} \label{eq:MomentPCE}
	\mean [Z] = \pce{z}^0,\quad \Sigma[Z,\tilde{Z}] = \displaystyle{\sum_{j=1}^{L-1}} \pce{z}^j\tilde{\pce{z}}^{j\top} \| \phi^j(\xi) \|^2.
\end{equation}
We denote $\Sigma[Z,Z]$ by the shorthand $\Sigma[Z]$.

\subsection{Problem Reformulation in PCE}
\begin{assum}[Exact PCEs for $\Xini$ and  $W_k$]\label{ass:ExactIniW}
	The initial condition $\Xini$ and all i.i.d. disturbances $W_{k}$, $k\in \I_{[0,N-1]}$ in OCP~\eqref{eq:StochOCP} admit exact PCEs, cf. Definition ~\ref{def:ExactPCE}, with $L\ini$ terms and $L_w$ terms, respectively. Precisely, $\Xini = \textstyle{\sum_{i=0}^{L\ini-1}} \pce{x}\ini^i \varphi^i(\xi\ini)$ and $W_k = \textstyle{\sum_{n=0}^{L_w-1}} \pce{w}_k^n \psi^n(\xi_k)$ for $k\in \I_{[0,N-1]}$,
where $\xi_k$ are i.i.d. stochastic germs. Note that $ \varphi^0(\xi\ini){=}\psi^0(\xi_k){=}1$, and $L\ini$, $L_w \in\N$.
\end{assum}
In the above assumption, each $\xi_k$, $k\in\I_{[0,N-1]}$ corresponds to the disturbance $W_k$ at time step $k$. Thus, $\{\xi_k\}_{k=0}^{N-1}$ is a stochastic process. To distinguish the sources of uncertainties acting on the system, we use $\varphi$ and $\psi$ to refer to the PCE basis for the initial condition $\Xini$ and, respectively, to the bases for the disturbances $W_k$, $\I_{[0,N-1]}$. In other words, the distributions of random variables are expressed by the algebraic structure of the basis functions and the corresponding germ $\xi$. The correlation between random variables is determined by the interplay of the coefficients, cf.~\eqref{eq:MomentPCE}, and stochastic  independence can be modelled by the use of different germs. To convey context in our notation, we employ the index variables $i$ and $n$ in the PCEs of $\Xini$ and of $W_k$, $k\in \I_{[0,N-1]}$, respectively. We define the bases
$\Phi^{\tini} \coloneqq \{\varphi^i(\xi\ini)\}_{i=0}^{L\ini-1}$ and $\Psi^{w_k}\coloneqq \{\psi(\xi_k)\}_{n=0}^{L_w-1}$.
That is, $\Phi^{\text{ini}}$ is the basis for $\Xini$ and $\Psi^{w_k}$ is the one for $W_k$ at time step $k$. Then we construct the joint basis 
\begin{equation} \label{eq:FiniteBasis}
	\Phi \coloneqq \Phi^{\text{ini}}\cup \Psi^w \quad \text{with}\quad \Psi^w \coloneqq \cup_{k=0}^{N-1}\Psi^{w_k},
\end{equation}
where $\Psi^w$ collects all bases $\Psi^{w_k}$, $k\in\I_{[0,N-1]}$ over the entire horizon. Therefore, $\Phi$ reads
\begin{multline} \label{eq:BasisElement}
\Phi = \Big\{ 1, \underbrace{ \varphi^1(\xi\ini),...,\varphi^{L\ini-1}(\xi\ini)}_{\Phi^{\tini}\setminus \{\varphi^0(\xi\ini)\}},\underbrace{\psi^1(\xi_0),...,\psi^{L_w-1}(\xi_0)}_{\Psi^{w_0}\setminus \{\psi^0(\xi_0)\}},
\\ ...,\underbrace{\psi^1(\xi_{N-1}),...,\psi^{L_w-1}(\xi_{N-1})}_{\Psi^{w_{N-1}}\setminus \{\psi^0(\xi_{N-1})\}}\Big\}.
\end{multline}
It contains a total of $L = L\ini+N(L_w-1)$ terms, i.e., it grows linearly with the horizon $N$. The following result is case (ii) of Proposition~1 by \citet{pan23stochastic}.

\begin{prop}[Exact uncertainty propagation] \label{pro:Basis}
Consider OCP~\eqref{eq:StochOCP} with horizon $N\in\N$ and let Assumption~\ref{ass:ExactIniW} hold. Suppose an optimal solution $\{U_k^\star\}_{k=0}^{N-1}$ to OCP~\eqref{eq:StochOCP} exists. Then $\{X_k^\star\}_{k=0}^N$ and $\{U_k^\star\}_{k=0}^{N-1}$ admit exact PCEs in the basis $\Phi$ from \eqref{eq:FiniteBasis}.
\end{prop}
For the sake of clarity, we make the following simplification. We discuss its generalization as well as the relaxation of Assumption~\ref{ass:ExactIniW} in Section~\ref{sec:Extensions}.
\begin{assum} \label{ass:Simplification}
The process disturbance $W_k$, $k\in\N$ is a scalar random variable, i.e. $n_w=1$. Furthermore, Assumption~\ref{ass:ExactIniW} is satisfied with $L\ini=L_w=2$.
\end{assum}
Now we enumerate the simplified joint basis $\Phi=\{\phi^j\}_{j=-2}^{N-1}$ as
\begin{equation}\label{eq:BasisSim}
	\begin{split}
		\phi^j = \begin{cases}
			1, &\text{for } j=-2\\
			\varphi^1(\xi\ini), &\text{for } j=-1\\
			\phi^1(\xi_j),&\text{for } j\in\I_{[0,N-1]}
		\end{cases},
	\end{split}
\end{equation}
which contains $L=N+2$ terms. Henceforth, we drop the stochastic germs $\xi$ of the basis function $\phi^j$, $j\in\I_{[-2,N-1]}$ in our notation. Indeed, the specific germ directly follows from the index $j$. The index of $\Phi$ starts with $-2$ such that the terms $\phi^j$, $j\in\I_{[0,N-1]}$, correspond to the PCE bases $\psi^1(\xi_j)$ of the disturbance $W_j$, respectively. In other words, the index $j\in\I_{[0,N-1]}$ directly corresponds to the time step at which the disturbance $W_j$ enters the problem. As we will see later, this particular indexing is helpful in revealing crucial structure of the PCE reformulation.

Moreover, we remark that in the union of the individual bases in~\eqref{eq:FiniteBasis}, only one constant basis function for the expected value is kept and this basis function is indexed with $j=-2$.
The orthogonality of the basis~$\Phi$ holds as $\Xini$ and $W_k$, $k\in\I_{[0,N-1]}$ are all independent, i.e., $\langle \phi^i,\phi^j\rangle = \mean[\phi^i]\mean[\phi^j]=0$, for all $ i,j \in \I_{[-2,N-1]},\,i\neq j$.

With Assumption~\ref{ass:Simplification}, we replace the random variables in the stage cost of OCP~\eqref{eq:StochOCP} with their PCEs and get
\begin{align*}
	&\ell(X_k,U_k)=\Big(\sum\nolimits_{j=-2}^{N-1}\pce{x}_k^{j\top}\phi^j\Big)Q\Big(\sum\nolimits_{j=-2}^{N-1}\pce{x}_k^{j}\phi^j\Big)+\\
	&\hspace{70pt}\Big(\sum\nolimits_{j=-2}^{N-1}\pce{u}_k^{j\top}\phi^j\Big)R\Big(\sum\nolimits_{j=-2}^{N-1}\pce{u}_k^{j}\phi^j\Big)\\
	=&\sum\nolimits_{j=-2}^{N-1}\big(\pce{x}_k^{j\top}Q\pce{x}_k^j + \pce{u}_k^{j\top}R\pce{u}_k^j\big)\|\phi^j\|^2
\end{align*} 
for all $k\in\I_{[0,N-1]}$ from the orthogonality~\eqref{eq:Orthogonality}.
Together with the dynamics of the PCE coefficients~\eqref{eq:SysPCE}, we arrive at the exact reformulation of \eqref{eq:StochOCP}
\begin{equation} \label{eq:StochPCE}
	\begin{split}
			\min_{\substack{\pce{u}_k^j \in\R^{n_u},\\k\in\I_{[0,N-1]},\\j\in\I_{[-2,N-1]}}}~ &  \sum_{j=-2}^{N-1} \Big(\|\pce{x}_N^j\|_{Q_N}^2 + \sum_{k=0}^{N-1}  \ell(\pce{x}_k^j,\pce{u}_k^j) \Big) \| \phi^j\|^2\\
		\text{s.t.}\quad& \eqref{eq:SysPCE},\quad j\in \I_{[-2,N-1]},
	\end{split}
\end{equation}
where  $\ell(\pce{x}_k^j,\pce{u}_k^j)= \|\pce{x}_k^j\|_Q^2+\|\pce{u}_k^j\|_R^2$. It is easy to see that OCP~\eqref{eq:StochPCE} entails $L=N+2$ decoupled optimization problems. Hence its solution is obtained by solving
\begin{equation} \label{eq:StochPCEj}
		\min_{\pce{u}_k^j \in\R^{n_u},  k\in\I_{[0,N-1]}}~ \|\pce{x}_N^j\|_{Q_N}^2 + \sum_{k=0}^{N-1} \ell(\pce{x}_k^j,\pce{u}_k^j)\quad \text{s.t.}~ \eqref{eq:SysPCE},
\end{equation}
separately for all $j \in \I_{[-2,N-1]}$. The key observation here is that each source of uncertainty in system~\eqref{eq:Sys}, i.e., the uncertain initial condition $\Xini$ and the disturbances $W_k$ at each time step $k\in\I_{[0,N-1]}$ can be decoupled and thus considered separately. The minimum of OCP~\eqref{eq:StochPCEj} for all $j\in\I_{[-2,N-1]}$ for the optimal input $\{\pce{u}_k^{j,\star}\}_{k=0}^{N-1}$ is written as $J_N(\pce{x}\ini^j,\pce{u}^{j,\star})$. From the optimal trajectory of the decoupled OCPs~\eqref{eq:StochPCEj}, one can compute the optimal trajectory of OCP~\eqref{eq:StochOCP} in random variables as $Z_k^\star=\sum_{j=-2}^{k-1} \pce{z}_k^{j,\star}\phi^j$, $k\in\I_{[0,N-1]}$ and $(Z,\pce{z})\in\{(X,\pce{x}),(U,\pce{u})\}$.

\subsection{Recap---The LQR for Affine Systems}
To finish the setup, we recall the deterministic LQR  for affine systems\citep{anderson89optimal}. Readers familiar with this material may jump directly to Section~\ref{sec:FiniteLQ}. Consider
\begin{equation}\label{eq:DeterOCP}
	\begin{split}
			\min_{ u_k\in\R^{n_u},k\in\I_{[0,N-1]} }~ &  \|x_N\|_{Q_N}^2 + \sum_{k=0}^{N-1} \ell(x_k,u_k)\\
		\text{s.t.} \quad x_{k+1} = &Ax_k + Bu_k + Ec, \quad x_0=x\ini,
	\end{split}
\end{equation}
where $c \in\R^{n_c}$ is a known constant and thus the dynamics are affine. Same to the stochastic counterpart, we denote the cost functional evaluated along $\{u_k^\star\}_{k=0}^{N-1}$ by $J_N(x_\tini,u^\star)$.
OCP~\eqref{eq:DeterOCP} can be written as
\begin{equation} \label{eq:OCPAugmented}
	\begin{split}
			\min_{ u_k\in\R^{n_u},k\in\I_{[0,N-1]} }~  & \|z_N\|_{Q_N^\prime}^2 + \sum_{k=0}^{N-1} \|z_k\|_{Q^\prime}^2 + \|u_k\|_R^2\\
		\text{s.t.}\quad z_{k+1} &= A^\prime z_k + B^\prime u_k,\quad	z_0=z_\tini 
	\end{split}
\end{equation}
with $z_{k} \coloneqq \big[\begin{smallmatrix*}[l] x_{k}\\ c \end{smallmatrix*}\big]$, $A^\prime \coloneqq \big[\begin{smallmatrix*}[l] A & E \\ 0_{n_c\times n_x} & 1_{n_c\times n_c} \end{smallmatrix*}\big]$, $B^\prime \coloneqq \big[\begin{smallmatrix*}[l] B \\ 0_{n_c\times n_u}\end{smallmatrix*}\big]$, $Q^\prime \coloneqq \text{blkdiag}(Q,0_{n_c\times n_c})$, $Q_N^\prime \coloneqq \text{blkdiag}(Q_N,0_{n_c\times n_c})$.
The optimal solution reads
$ u_k^{\star} = K_{N-k}^\prime z_k^\star$ with $ K_{k}^\prime = -(R+B^{\prime \top} P_{k-1}^\prime B^\prime)^{-1}B^{\prime \top} P_{k-1}^\prime A^\prime$. The matrix $P_k^\prime$ is computed by $P_0^\prime=Q_N^\prime$, and the Riccati difference equation $P_{k+1}^\prime = Q^\prime+A^{\prime \top}\Big(P_{k}^\prime - P_{k}^\prime B^\prime (R+B^{\prime \top} P_{k}^\prime B^\prime)^{-1} B^{\prime \top} P_{k}^\prime \Big)A^\prime$. Consider
$P_k^\prime \coloneqq \big[\begin{smallmatrix*}[l]
P_k & G_k \\ G_k^\top & S_k
\end{smallmatrix*}\big]$, $k\in\I_{[0,N]}$ with $P_k\in\R^{n_x \times n_x}$, $G_k\in\R^{n_x\times n_c}$, and $S_k\in\R^{n_c\times n_c}$.
Then we have
\begin{subequations} \label{eq:SolutionDeterLQR}
	\begin{equation}
		u_k^{\star} = K_{N-k} x_k^\star + F_{N-k} c,
	\end{equation}
	\parbox{\linewidth}{where $x_k^\star$ is the optimal state, $K_{k} = -M_{k-1}^{-1}B^\top P_{k-1}A$, $F_{k} = {-}M_{k-1}^{-1}B^\top \big( P_{k-1}E {+} G_{k-1} \big)$, and $M_{k-1} = R+B^\top P_{k-1} B$. $P_k$ and $G_k$ are recursively computed by $P_0 =Q_N$, $G_0=0$, and}
	\begin{align}
		P_{k} &= Q{+}A^\top \Big(P_{k-1}{-} P_{k-1} B M_{k-1}^{-1} B^\top P_{k-1} \Big)A, \label{eq:RiccatiP}\\
		G_{k} &= (A+BK_{k})^\top (P_{k-1}E+G_{k-1}). \label{eq:RiccatiG}
	\end{align}
\end{subequations}
The feedback~\eqref{eq:SolutionDeterLQR} is a simplified case of Theorem~1 by \citet{singh17extended}, where $c$ is not constant. The minimum cost is
\begin{multline} \label{eq:SolutionDeterCost}
	J_N(x\ini,u^\star) =  \|x_N^\star\|_{Q_N}^2 + \textstyle{\sum_{k=0}^{N-1}} \ell(x_k^\star, u_k^\star) \\
	= z_0^\top P_N^\prime z_0 =  x\ini^{\top}P_N x\ini + 2c^\top G_N^\top x\ini + c^\top S_N c,
\end{multline}
where $S_N$ is computed by $S_0=0$ and $	S_{k} = S_{k-1} + E^\top G_{k-1} + G_{k-1}^\top E	+ E^\top P_{k-1}E - F_{k}^\top M_{k-1}F_{k}$.

Now we turn towards OCP~\eqref{eq:DeterOCP} with infinite horizon $N=\infty$ and $Q_N=0$. We use the superscript $\cdot^\sinf$ to highlight the infinite-horizon optimal solution to OCP~\eqref{eq:DeterOCP}. We also note that the stage cost $\ell(x,u)$ might be non-zero for affine systems, which in turn leads to an unbounded objective in the infinite-horizon OCP. Standard notions of optimality cannot be applied in this case. Hence, we recall the concept of overtaking optimality from \citet{carlson91infinite} for deterministic and stochastic OCPs.
\begin{defn}[Overtaking optimality] \label{def:OvertakingDeter}
Consider OCP~\eqref{eq:DeterOCP} with $N=\infty$ and $Q_N=0$. The control sequence $\{u_k^\sinf\}_{k=0}^\infty$ is deterministically overtakingly optimal if, for any other $\{u_k\}_{k=0}^\infty$, we have
\[
\liminf_{N\to\infty} J_N(x\ini, u) - J_N(x\ini,u^\sinf) \geq 0.
\]
Additionally, for the stochastic OCP~\eqref{eq:StochOCP} with $N=\infty$ and $Q_N=0$, the control sequence $\{U_k^\sinf\}_{k=0}^\infty$ is stochastically overtakingly optimal if, for any other $\{U_k\}_{k=0}^\infty$, it holds that 
\[
	\liminf_{N\to\infty} J_N(X\ini, U) - J_N(X\ini,U^\sinf) \geq 0.
\]
\end{defn}
Extending~\eqref{eq:SolutionDeterLQR} to the infinite-horizon case, the overtakingly optimal feedback for OCP~\eqref{eq:DeterOCP} with $N=\infty$ and $Q_N=0$ is given by
\begin{equation} \label{eq:SolutionDeterLQRInf}
		u_k^{\sinf} = K x_k^\sinf + F c,
\end{equation}
where $K$ and $F$ are stationary solutions to~\eqref{eq:SolutionDeterLQR}.
\section{Stochastic LQR on Finite Horizon} \label{sec:FiniteLQ}
\subsection{Solution in PCE Coefficients}\label{subsec:SolutionPCE}
First we rewrite the PCE of $\Xini$ and $W_k$, $k\in \I_{[0,N-1]}$ in the joint basis $\Phi$ from \eqref{eq:FiniteBasis}. Let Assumption~\ref{ass:Simplification} hold, and let the PCEs of $\Xini$ and $W_k$, $k\in \I_{[0,N-1]}$ in basis $\Phi$  be $\Xini = \sum_{j=-2}^{N-1} \pce{x}\ini^j \phi^j$ and $W_k = \sum_{j=-2}^{N-1} \pce[j]{w}_k\phi^j$, respectively. Then we have
\begin{subequations} \label{eq:PCECoe}
	\begin{align}
		\pce{x}\ini^j &= 0,\hspace{28pt}\forall j\in \I_{[0,N-1]},\label{eq:PCEXini}\\
		\pce{w}_k^{-1} &= \pce{w}_k^{j} =0,~\forall j,k \in \I_{[0,N-1]},~j\neq k, \label{eq:PCEWkj}\\
		\pce{w}_0^{-2} &= \pce{w}_1^{-2} = ... = \pce{w}_{N-1}^{-2} \coloneqq \mean[W], \label{eq:PCEW0}\\
		\pce{w}_0^0 &= \pce{w}_1^1 = ... = \pce{w}_{N-1}^{N-1}\coloneqq \pce{w}^0. \label{eq:PCEWw}
	\end{align}
\end{subequations}
Note that \eqref{eq:PCEXini}-\eqref{eq:PCEWkj} follow from the independence of the random variables $X\ini$ and $W_k$ and from the considered basis indexing. Equations \eqref{eq:PCEW0}-\eqref{eq:PCEWw} are due to $W_k$, $k\in\I_{[0,N-1]}$ being identically distributed.

\begin{lem}[Optimal solution via PCE] \label{lem:SolutionPCEj}
Consider OCP~\eqref{eq:StochPCEj} for all $j\in\I_{[-2,N-1]}$ and let Assumption~\ref{ass:Simplification} hold. For $k\in\I_{[0,N-1]}$, the optimal input is
\begin{subequations}
	\begin{equation} \label{eq:SolutionPCE}
		\pce{u}_k^{j,\star} = \begin{cases}  K_{N-k}\pce{x}_k^{j,\star} + F_{N-k} \mean[W], &\text{for } j=-2\\
			K_{N-k}\pce{x}_k^{j,\star}, &\text{otherwise}
		\end{cases}
	\end{equation}
	\parbox{\linewidth}{with $K_{N-k}$ and $ F_{N-k}$ from \eqref{eq:SolutionDeterLQR}. It yields the minimum cost $J_N(\pce{x}\ini^j,\pce{u}^{j,\star})=$}
	\begin{equation} \label{eq:SolutionCost}
			\begin{cases}  \|\mean[X\ini]\|_{P_N}^2 + \|\mean[W]\|_{S_N}^2 \\
				\hspace{40pt}+ 2\mean[W]^\top G_N^\top\mean[X\ini], &\text{for } j=-2\\
				\tr(P_N\Sigma[X\ini])/\|\phi^{-1}\|^2, &\text{for } j=-1\\
				\tr(P_{N-j-1}E\Sigma[W]E^\top)/\| \phi^j\|^2, &\text{otherwise}
			\end{cases}.
	\end{equation}
\end{subequations}
\end{lem}
\begin{proof}
The proof proceeds in three steps.  Step I)---Propagation of the expected value. Consider OCP \eqref{eq:StochPCEj} for $j=-2$ given by
\begin{equation} \label{eq:StochPCEExp}
	\begin{split}
		 &\min_{\pce{u}_k^{-2}\in\R^{n_u},~k\in\I_{[0,N-1]}}~   \|\pce{x}_N^{-2}\|_{Q_N}^2+ \displaystyle{\sum_{k=0}^{N-1}} \ell(\pce{x}_k^{-2},\pce{u}_k^{-2}) \\
		&\text{s.t.}\quad \pce{x}_{k+1}^{-2} = A\pce{x}_k^{-2} +B\pce{u}_k^j+ E\pce{w}_k^{-2},\quad \pce{x}_{0}^{-2} = \pce{x}^{-2}\ini.
	\end{split}
\end{equation}
As $\pce{w}_k^{-2} =  \mean[W]$, $k\in\I_{[0,N-1]}$ is constant over time, \eqref{eq:SolutionDeterLQR} suggests the optimal feedback
$\pce{u}_k^{{-2},\star} = K_{N-k}\pce{x}_k^{{-2},\star} + F_{N-k} \mean[W]$. Then \eqref{eq:SolutionCost} for $j=-2$ follows from \eqref{eq:SolutionDeterCost}.

Step II)---Propagation of the non-mean part of the initial condition. Since $\pce{w}_k^{-1} =0$, $\forall k \in \I_{[0,N-1]}$ holds, OCP \eqref{eq:StochPCEj} for $j=-1$ is simplified as
\begin{equation} \label{eq:StochPCEIni}
	\begin{split}
		\min_{\pce{u}_k^{-1},~k\in\I_{[0,N-1]}}~  &\|\pce{x}_N^{-1}\|_{Q_N}^2+\displaystyle{\sum_{k=0}^{N-1}} \ell(\pce{x}_k^{-1},\pce{u}_k^{-1}) \\
		\text{s.t.}\quad \pce{x}_{k+1}^{-1} &= A\pce{x}_k^{-1} + B\pce{u}_k^{-1},\quad \pce{x}_0^{-1}=\pce{x}\ini^{-1}.
	\end{split}
\end{equation}
We observe that OCPs~\eqref{eq:StochPCEj}, $\forall j\in\I_{[-2,N-1]}$ share the same weighting matrices $Q_N$, $Q$, $R$ and the same system matrices $A$, $B$. Therefore, we obtain
$\pce{u}_k^{-1,\star} = K_{N-k}\pce{x}_k^{-1,\star}$ and
$J_N(\pce{x}\ini^{-1},\pce{u}^{-1,\star})=\pce{x}\ini^{-1\top}P_N \pce{x}\ini^{-1}$.

Step III)---Propagation of the non-mean part of the disturbances. Consider the dynamics of the PCE coefficients for all $j\in\I_{[0,N-1]}$. The causality requirement stemming from the consideration of the adapted filtration for $U_k \in \lsp_{k}(\R^\dimu)$ implies that $U_k$ only depends on $X_i$, $i\leq k$. Due to our chosen indexing and the causality, we obtain $\pce{w}_k^j=0$ and $\pce{u}_k^j = 0$ for $k\leq j$, which implies $\pce{x}_k^j = 0$, $k\leq j$. We observe that $\pce{x}_{j+1}^{j} = A\cdot 0 + B\cdot 0 + E\pce{w}_j^j=E\pce{w}^0$ and $\pce{w}_k^{j}=0$, $k\geq j+1$ as \eqref{eq:PCEWkj} and \eqref{eq:PCEWw} hold. Therefore, an equivalent reformulation of OCP~\eqref{eq:StochPCEj}, $j\in\I_{[0,N-1]}$ is
\begin{equation} \label{eq:StochPCEW}
	\begin{split}
		\min_{\pce{u}_k^{j},~k\in\I_{[0,N-1]}}~ &\|\pce{x}_N^{j}\|_{Q_N}^2+\displaystyle{\sum_{k=j+1}^{N-1}} \ell(\pce{x}_k^{j},\pce{u}_k^{j})\\
		\text{s.t.} \quad 		\pce{x}_{k+1}^{j} &= A\pce{x}_k^{j} + B\pce{u}_k^{j},~k\geq j+1, \\ 
		\pce{x}_{j+1}^{j} &= E\pce{w}^0, \quad \pce{x}_{k}^{j} = 0,~k\leq j.
	\end{split}
\end{equation}
For $k\geq j+1$, the optimal feedback for~\eqref{eq:StochPCEW} is $\pce{u}_k^{j,\star} = K_{N-k}\pce{x}_k^{j,\star}$. We can extend it to the case $k\in\I_{[0,N-1]}$ since $\pce{u}_k^{j,\star}=\pce{x}_k^{j,\star}=0$ for $k\in\I_{[0,j]}$. The minimum for $j\in\I_{[0,N-1]}$ is $J_N(\pce{x}\ini^j,\pce{u}^{j,\star})= \pce{w}^{0\top}E^\top P_{N-j-1}E\pce{w}^0=\tr(P_{N-j-1}E\Sigma[W]E^\top)/\| \phi^j\|^2$.
\end{proof}

\subsection{Optimal state trajectories in PCE} \label{subsec:Illustration}
Applying the feedback~\eqref{eq:SolutionPCE} to \eqref{eq:SysPCE}, we obtain the optimal state trajectories of PCE coefficients.
\begin{prop}[PCE coefficient trajectories] \label{prop:PCETraj}
Consider OCP~\eqref{eq:StochPCEj} for all $ j\in\I_{[-2,N-1]}$. The optimal state trajectories of PCE coefficients are
\begin{equation}   \label{eq:PCEXTraj}
	\pce{x}_k^{j,\star} {=} \begin{cases}
		\bar{A}_0^{k-1}\pce{x}\ini^{-2}{+}\sum_{i=0}^{k-1}\bar{A}_{i+1}^{k-1}\tilde{F}_{N-i}\mean[W],  &\text{for } j=-2\\
		\bar{A}_0^{k-1}\pce{x}\ini^{-1},  &\text{for } j=-1\\
			0,\hfill \text{for } k \leq j,\hspace*{16pt}&j\in\I_{[0,N-1]}\\
			\bar{A}_{j+1}^{k-1} E\pce{w}^0,\hfill\text{for } k\geq j+1,&j\in\I_{[0,N-1]}
	\end{cases}
\end{equation}
with $\tilde{F}_{N-i} \coloneqq BF_{N-i}+E$. Note that $\pce{w}^0=\pce{w}_j^j$ holds for $j\in\I_{[0,N-1]}$. For all $k_1,k_2\in\I_{[0,N-1]}$, let the matrix $\bar{A}_{k_1}^{k_2}$ be
\[
\bar{A}_{k_1}^{k_2} \coloneqq \begin{cases} \prod_{k=k_1}^{k_2} (A+BK_{N-k}),&\text{for } 0 \leq k_1 \leq k_2 \\
	I,&\text{otherwise} \end{cases}.
\]
Then for the PCE coefficients related to the disturbances, i.e. for all $k\in\I_{[1,N]}$ and $j\in\I_{[0,k-1]}$, we have
\begin{subequations}\label{eq:Observation}
	\begin{itemize}
		\item for fixed PCE coefficient dimension $j$
		\begin{equation}
			\pce{x}_{k+t}^{j,\star} = \bar{A}_{k}^{k+t-1}\pce{x}_k^{j,\star}, \quad t\in\I_{[0,N-k]};
		\end{equation}
		\item for fixed time step $k$ over PCE coefficient dimension
		\begin{equation} \label{eq:Observation2}
			\pce{x}_k^{j-t,\star} = \bar{A}_{j-t+1}^{j}\pce{x}_k^{j,\star},\quad t \in\I_{[0,j]}.
		\end{equation}
	\end{itemize}
\end{subequations}
\end{prop}
\begin{proof}
Plugging the optimal feedback~\eqref{eq:SolutionPCE} into the PCE coefficient dynamics~\eqref{eq:SysPCE}, one obtains $\{\pce{x}_k^{j,\star}\}_{k=0}^N$ for $j\in\I_{[-2,N-1]}$, which is illustrated in Figure~\ref{fig:PCEIllustration}. 
	
For fixed PCE dimension $j\in\I_{[0,k-1]}$, we obtain $\pce{x}_{k+t}^{j,\star} = \bar{A}_{j+1}^{k+t-1} E\pce{w}^0 = \bar{A}_{k}^{k+t-1}\bar{A}_{j+1}^{k-1} E\pce{w}^0 = \bar{A}_{k}^{k+t-1}\pce{x}_k^{j,\star}$, $t\in\I_{[0,N-k]}$, cf. Figure~\ref{fig:PCEIllustration}. 
 
Moreover, we freeze the time step $k\in\I_{[1,N]}$ and project the points $\pce{x}_k^{j,\star}$, $j\in\I_{[0,k-1]}$ onto the $\pce{x}_k^j-j$~plane in Figure~\ref{fig:PCEIllustration}. Each line with circle markers in the projection represents $\{\pce{x}_k^{j,\star}\}_{j=0}^{k-1}$ wherein $j$ is the running index for fixed $k$.
Observe that the structure of OCP~\eqref{eq:StochPCEW} links the PCE coefficients for fixed time step $k$. Specifically, we have that $\pce{x}_k^{j-t,\star} = \bar{A}_{j-t+1}^{k-1} E\pce{w}^0 = \bar{A}_{j-t+1}^{j} \bar{A}_{j+1}^{k-1} E\pce{w}^0= \bar{A}_{j-t+1}^{j} \pce{x}_k^{j,\star}$ for $t \in\I_{[0,j]}$.
\end{proof}

\begin{figure}[t]
	\begin{center}
		\includegraphics[width=0.5\textwidth]{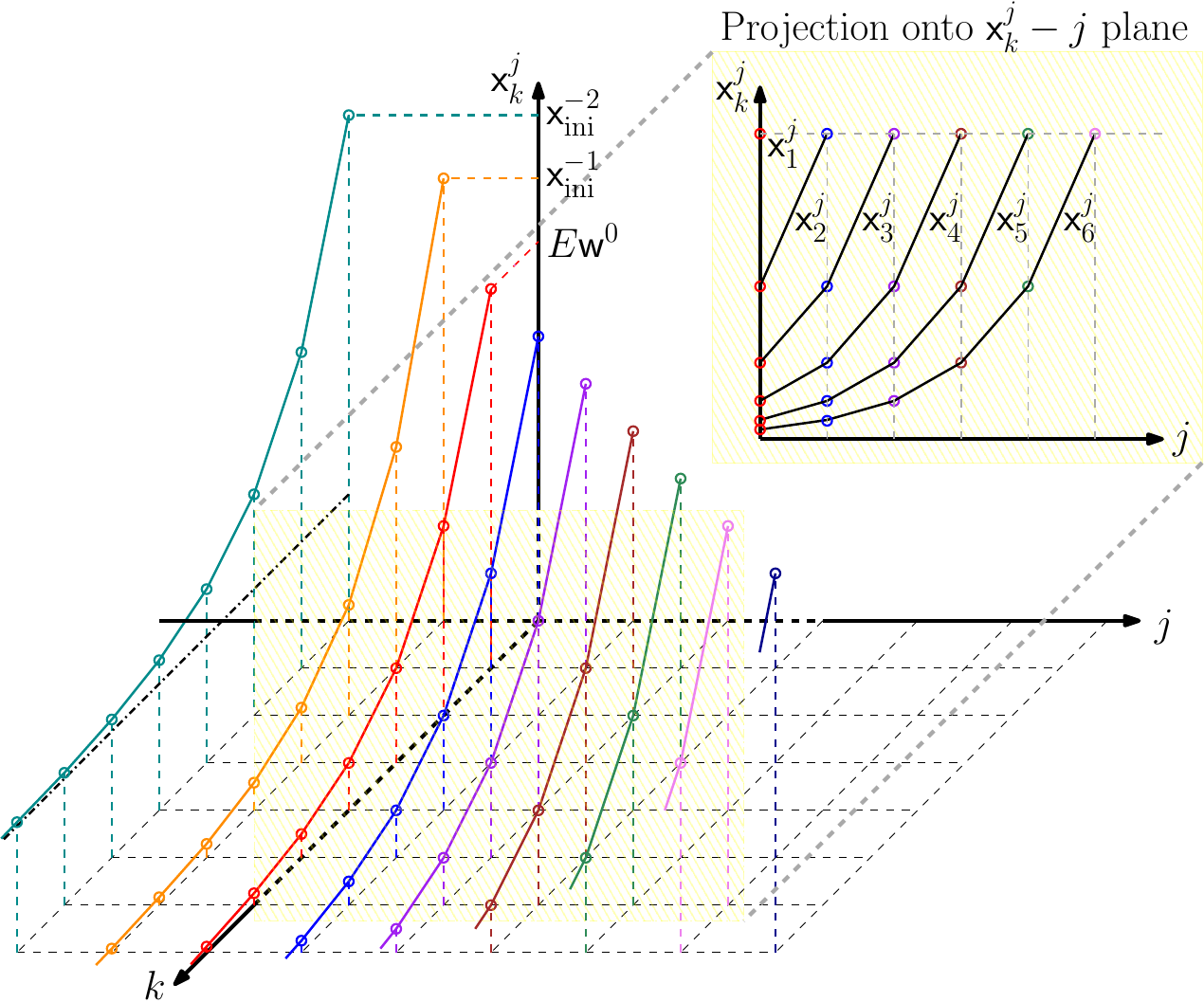}
		\caption{Optimal trajectories of OCP~\eqref{eq:StochPCEj} in PCE coefficients $\pce{x}^{j,\star}$, $j\in\I_{[-2,N-1]}$. \label{fig:PCEIllustration}} 		
	\end{center}
\end{figure}

Figure~\ref{fig:PCEIllustration} illustrates the core idea behind the crucial insight~\eqref{eq:Observation} of the previous result. One can see that $\{\pce{x}_k^{j,\star}\}_{k=0}^{N}$ for any fixed $j\in \I_{[-2,N-1]}$ converges to its corresponding steady state over time. Depending on the system dynamics and on the weighting matrices, there are also potentially leaving arcs at the end of the trajectories, which is related to the turnpike phenomenon, see \citet{faulwasser22turnpike}. Additionally, one sees that $\pce{x}_{j+1}^{j,\star}=E\pce{w}^0$ for all $j\in\I_{[0,N-1]}$, i.e., the trajectories $\{\pce{x}_{k}^{j,\star}\}_{k=j+1}^{N}$, $j\in\I_{[0,N-1]}$ have the same initial value $E\pce{w}^0$ at time step $j+1$. Equation~\eqref{eq:Observation2} shows that for fixed time index~$k$, $\pce{x}_k^{j,\star}$ decays as $j$ decreases. This is in line with the intuition that the most recent disturbances are dominant in the PCE description of the state variable $X_k^\star$.

\subsection{Moving-Horizon PCE Series Truncation}
Proposition~\ref{pro:Basis} suggests that the dimension $L$ of the joint basis $\Phi$ grows linearly with the horizon $N$ due to the process disturbances. To accelerate the computation in numerical implementations, it is often desirable to truncate the PCE. As Figure~\ref{fig:PCEIllustration} indicates, to minimize the truncation error at time step $k$, we may only consider the basis related to the initial condition and to the last $p$ disturbances. That is, we consider the $(p+2)$-dimensional truncated basis $\Phi_k^{\text{trun}} = \Phi^{\tini}\cup\big(\cup_{\tilde{k} = k-p}^{k-1} \Psi^{w_{\tilde{k}}}\big)$. Figure~\ref{fig:Truncation} illustrates the PCE coefficients in the truncated basis $\Phi^{\text{trun}}$ with $p=2$ as an example, where the red box includes the PCE coefficients of $\Phi^{\tini}$, while the blue boxes include the PCE coefficients of $\cup_{\tilde{k} = k-p}^{k-1}\Psi^{w_{\tilde{k}}}\setminus \{\phi^{-2}\}$.
As the optimal trajectories of PCE coefficients are known, we next quantify the error stemming from this moving-horizon series truncation. Moreover, a result related to the upper bound of the truncation error will be shown in Lemma~\ref{lem:FiniteApprox}, Section~\ref{sec:Stationary}.
\begin{lem}[Quantification of truncation errors] \label{lem:Truncation}
Let Assumption~\ref{ass:Simplification} hold. Consider OCP~\eqref{eq:StochOCP} and the truncated moving-horizon basis $\Phi_k^{\text{trun}}$. Then the truncation error $\Delta X_k(p+2)\coloneqq X_k^{\star} - X_k^{\text{trun},\star}$ at time step $k\in \I_{[0,N-1]}$~reads
\[
\Delta X_k(p+2) = \begin{cases} 0, &\text{for } k\leq p \\
	\sum_{j=0}^{k-p-1} \bar{A}_{j+1}^{k-1}E\pce{w}^0\phi^j, &\text{otherwise }\end{cases},
\]
where the argument $(p+2)$ refers to the dimension of $\Phi_k^{\text{trun}}$,  and $X^{\text{trun},\star}$ is the random variable obtained from the PCE solution in the basis $\Phi^{\text{trun}}$.
\end{lem}
\begin{proof}
As the reformulated OCP~\eqref{eq:StochPCE} can be solved separately with respect to each PCE dimension~$j\in\I_{[-2,N-1]}$, we obtain $X_k^{\text{trun},\star}{=} \sum_{j\in\{-2,-1,k-p,...,k-1\}} \pce{x}_k^{j,\star}\phi^j$. For the case $k\leq p$, there is no truncation error. Due to the trajectories~\eqref{eq:PCEXTraj}, the error for all $k \geq p+1$ reads $ X_k^{\star} {-} X_k^{\text{trun},\star} {=} \sum_{j=0}^{k-p-1} \pce{x}_k^{j,\star}\phi^j {=}\sum_{j=0}^{k-p-1} \bar{A}_{j+1}^{k-1}E\pce{w}^0\phi^j$.
\end{proof}
\begin{figure}[t]
	\begin{center}
		\includegraphics[width=0.45\textwidth]{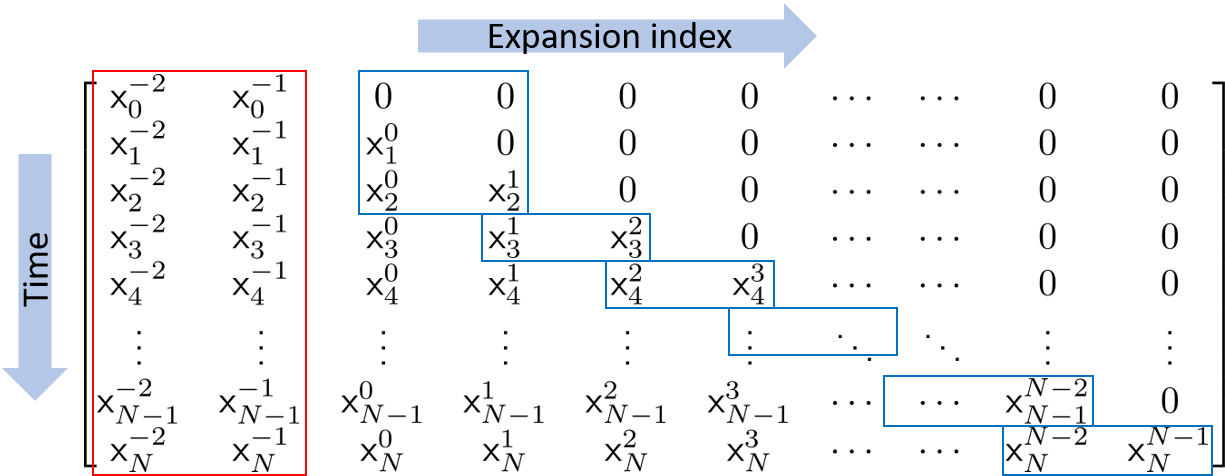}
		\caption{Truncation of PCE coefficients.} \label{fig:Truncation}		
	\end{center}
\end{figure}

\subsection{Solution in Random Variables}
Leveraging the solution to OCP~\eqref{eq:StochPCEj} for all $j\in\I_{[-2,N-1]}$, we obtain the solution to the original OCP~\eqref{eq:StochOCP}.
\begin{thm}[Random variable solution] \label{theorem:SolutionFinite}
Let Assumption~\ref{ass:Simplification} hold. The unique solution to OCP~\eqref{eq:StochOCP} is
\begin{subequations}
\begin{equation} \label{eq:FiniteInput}
	U_k^\star = K_{N-k} X_k^\star + F_{N-k} \mean[W],
\end{equation}
\text{while the corresponding minimum cost reads}
\begin{multline}\label{eq:FiniteMinimumCost}
	J_N(X\ini,U^\star)=  \|X\ini\|_{P_N}^2 + 2\mean[W]^\top G_N^\top\mean[X\ini] \\
	+ \tr(\textstyle{\sum_{j=0}^{N-1}}P_j E\Sigma[W]E^\top) +  \|\mean[W]\|_{S_N}^2.
	\end{multline}
\end{subequations}
\end{thm}
\begin{proof}
	The feedback~\eqref{eq:FiniteInput} immediately follows from $U_k^\star= \sum_{j=-2}^{k-1}\pce{u}_k^{j,\star}\phi^{j}$ and Lemma~\ref{lem:SolutionPCEj}, i.e.,
	\begin{align*}
			U_k^\star &=(K_{N-k}\pce{x}_k^{-2,\star}{+}F_{N-k} \mean[W])\phi^{-2}{+}\textstyle{\sum_{j=-1}^{k-1}}K_{N-k}\pce{x}_k^{j,\star}\phi^j \\
		&= K_{N-k} \big(\textstyle{\sum_{j=-2}^{k-1}} \pce{x}_k^{j,\star}\phi^{j} \big)+ F_{N-k} \mean[W],
	\end{align*}
	where $\sum_{j=-2}^{k-1} \pce{x}_k^{j,\star}\phi^{j}=X_k^\star$. It remains to prove the uniqueness of~\eqref{eq:FiniteInput}. First Proposition~\ref{pro:Basis} shows that any optimal solution to OCP~\eqref{eq:StochOCP} lives in the space spanned by the joint basis~\eqref{eq:FiniteBasis}.
	Since the solution to OCP~\eqref{eq:StochPCEj} for all $j\in\I_{[-2,N-1]}$ is unique,
	$J_N(\pce{x}\ini^j,\pce{u}^{j,\star}) < J_N(\pce{x}\ini^j,\pce{u}^{j})$
	holds for any feasible $\pce{u}^{j}\neq \pce{u}^{j,\star}$. Then, for any $U=\sum_{j=-2}^{N-1}\pce{u}^{j}\phi^j \neq \sum_{j=-2}^{N-1}\pce{u}^{j,\star}\phi^j=U^\star $, it follows
	\begin{multline*}
		J_N(\Xini,U^\star) = \textstyle{\sum_{j=-2}^{N-1}} J_N(\pce{x}\ini^j,\pce{u}^{j,\star}) \|\phi^j\|^2\\
		< \textstyle{\sum_{j=-2}^{N-1}} J_N(\pce{x}\ini^j,\pce{u}^{j}) \|\phi^j\|^2 = J_N(\Xini,U).
	\end{multline*}
	We conclude the uniqueness of the solution~\eqref{eq:FiniteInput}. As $J_N(\pce{x}\ini^j,\pce{u}^{j,\star})$, $j\in\I_{[-2,N-1]}$ has been given by \eqref{eq:SolutionCost} in Lemma~\ref{lem:SolutionPCEj}, we also obtain the minimum cost~\eqref{eq:FiniteMinimumCost} using that $\|X\ini\|_{P_N}^2=\|\mean[X\ini]\|_{P_N}^2 +\tr(P_N \Sigma[X\ini])$.
\end{proof}
The optimal feedback~\eqref{eq:FiniteInput} is a well-known result, especially for stochastic LTI system with zero-mean Gaussian disturbances \citep{astrom70introduction,anderson79optimal}. Deviating from the moment-based approach, Theorem~\ref{theorem:SolutionFinite} computes the solution to OCP~\eqref{eq:StochOCP} directly in random variables and generalizes the results to non-Gaussian uncertainties with non-zero mean. With trajectories of PCE coefficients shown in Proposition~\ref{prop:PCETraj}, PCE allows to compute the optimal state trajectories in random variables in the non-Gaussian setting with constructive error analysis, cf. Lemma~\ref{lem:Truncation}.

\subsection{Relaxation of Assumptions} \label{sec:Extensions}
The reader may ask if and how Assumptions~\ref{ass:ExactIniW} and~\ref{ass:Simplification} can be relaxed for Theorem~\ref{theorem:SolutionFinite}. Indeed, the answer is affirmative for both assumptions. First we consider to drop Assumption~\ref{ass:Simplification} and let Assumption~\ref{ass:ExactIniW} still hold. That is, we consider a generalized case of $L\ini>2$ or $L_w>2$ with $n_w\geq1$ and $L\ini$, $L_w\in\N$. Similarly, we construct the joint basis $\Phi=\Phi^{\tini}\cup\Psi^w$, which has been given in \eqref{eq:BasisElement} element-wise. Since the dynamics~\eqref{eq:SysPCE} and the weighting matrices~\eqref{eq:StochInfPCEj} are the same for all PCE coefficients, one sees that under a suitable basis indexing a result similar to Lemma~\ref{lem:SolutionPCEj} can be obtained. Thus, Theorem~\ref{theorem:SolutionFinite} is still valid.

One can further extend the results to the case of $L\ini=\infty$ or $L_w=\infty$, i.e., when Assumption~\ref{ass:ExactIniW} is dropped. In this case, we construct the joint basis $\Phi$ in the same way and have $L=L\ini+N(L_w-1)=\infty$. Therefore, the outer sum in OCP~\eqref{eq:StochPCE} over PCE coefficients is $j\in\I_{[-2,\infty)}$, while system dynamics~\eqref{eq:SysPCE} remain the same for each PCE dimension $j$. This way, we have the same decoupled OCP~\eqref{eq:StochPCEj} in terms of PCE coefficients, which validates the optimal solution via PCE in Lemma~\ref{lem:SolutionPCEj}. Hence, we obtain the same optimal feedback~\eqref{eq:FiniteInput} and minimum cost~\eqref{eq:FiniteMinimumCost} in Theorem~\ref{theorem:SolutionFinite}. 

Consider system~\eqref{eq:Sys} and let the disturbances $W_k$, $k\in\N$ be independent but not identically distributed. Moreover, the distributions of $W_k$,~$k\in\N$ are assumed to be known. Hence, the PCE coefficients of the disturbances are also known in advance. For each disturbance $W_k$, $k\in\N$, there is a specific corresponding basis function in the joint basis~\eqref{eq:BasisElement}. Compared to Lemma~\ref{lem:SolutionPCEj}, the solution in the PCE coefficients, $j\in\N$, remain valid. The solution for $j=-2$ reads
\[
	\pce{u}_k^{j,\star} = K_{N-k}\pce{x}_k^{j,\star} + F_{N-k} \mean[W_k] +F_{N-k}^\prime.
\]
By treating $\mean[W_k]$, $k\in\I_{[0,N-1]}$ as exogenous inputs, the computation of $F_{N-k}^\prime$ proceeds in the same manner as Theorem~1 by \citet{singh17extended} and is thus omitted. Then the optimal feedback in random variables is of the form $U_k^{\star} = K_{N-k}X_k^{\star} + F_{N-k} \mean[W_k] +F_{N-k}^\prime$.
\section{Stochastic LQR of Infinite Horizon and Its Asymptotics} \label{sec:InfiniteLQ}
In this section, we extend the obtained results in Section~\ref{sec:FiniteLQ} to infinite horizon and analyze the convergence property of the infinite-horizon optimal trajectories.
\subsection{Stochastic LQ Optimal Control of Infinite Horizon} \label{sec:SolutionStochLQRInf}
The infinite-horizon counterpart to OCP~\eqref{eq:StochOCP} reads
\begin{equation}\label{eq:StochOCPInf}
		\min_{ U_k \in \lsp_{k}(\R^\dimu),k\in\N^\infty}  ~  \sum_{k=0}^{\infty}  \ell(X_k,U_k)\quad \text{s.t.}\quad \eqref{eq:Sys},
\end{equation}
where the terminal penalty is dropped. The cost functional of OCP~\eqref{eq:StochOCPInf} along $\{U_k\}_{k=0}^\infty$ is denoted by $J_\infty(X\ini,U)$. The covariance propagation
\[
\Sigma[X_{k+1}] {=} \Sigma[AX_k+BU_k] {+} E\Sigma[W_k]E^\top {\succeq} E\Sigma[W_k]E^\top {\succeq} 0
\]
implies that the stage cost
\begin{equation*}
	\ell(X_k,U_k) = \|\mbb E[X_k]\|_Q^2+\|\mbb E[U_k]\|_R^2 + \text{tr}(Q\Sigma[X_k]+R\Sigma[U_k])\geq0
\end{equation*}
over the infinite horizon. Therefore, the minimum cost may be infinite and we need to invoke the notion of overtaking optimality, cf. Definition~\ref{def:OvertakingDeter}.

Similar to \eqref{eq:FiniteBasis}, we construct the basis $\Phi_\infty \coloneqq\Phi^{\text{ini}}\cup \Psi_\infty^w$ with $\Psi_\infty^w \coloneqq \cup_{k=0}^{\infty}\Psi^{w_k}$. We enumerate it similar to~\eqref{eq:BasisSim} as $\Phi_\infty=\{\phi_\infty^j\}_{j=-2}^\infty$. Compared to the basis $\Phi$ for OCP~\eqref{eq:StochOCP} with $N\in\N$, $\Phi_\infty$ appends the basis for $W_k$, $k\geq N$ at the end. That is, $\phi_\infty^j=\phi^j$ holds for all $j\in\I_{[-2,N-1]}$ and $\Phi_\infty\setminus\Phi = \{\phi_\infty^j\}_{j=N}^\infty=\cup_{j=N}^\infty\{\psi^1(\xi_j)\}$. Therefore, we omit the subscript $\cdot_\infty$ of $\phi_\infty^j$. Reformulation of OCP~\eqref{eq:StochOCPInf} in the basis $\Phi_\infty$ gives for $j\in\I_{[-2,\infty)}$
\begin{equation} \label{eq:StochInfPCEj}
	\min_{\pce{u}_k^j \in\R^{n_u}, k\in\I_{[0,\infty)}}~  \displaystyle{\sum_{k=0}^{\infty}} \ell(\pce{x}_k^j,\pce{u}_k^j) \quad \text{s.t.} \quad \eqref{eq:SysPCE}.
\end{equation}
Recall that the superscript $\cdot^\sinf$ denotes the optimal solutions to OCPs with infinite horizon.
\begin{lem}[Infinite-horizon optimal solution] \label{lem:Infinite}
Consider OCP~\eqref{eq:StochInfPCEj} for $j\in\I_{[-2,\infty)}$ and let Assumption~\ref{ass:Simplification} hold. Then, for all $j\in\I_{[-2,\infty)}$, the unique overtakingly optimal solution is
\begin{equation*} 
	\pce{u}_k^{j,\sinf} = \begin{cases}  K\pce{x}_k^{j,\sinf} + F\mean[W], &\text{for } j=-2\\
		K\pce{x}_k^{j,\sinf}, &\text{otherwise}
	\end{cases}.
\end{equation*}
Hence, the unique overtakingly optimal solution to OCP~\eqref{eq:StochOCPInf} is $U_k^\sinf =KX_k^\sinf + F \mean[W]$.
\end{lem}
\begin{proof}
	Due to Assumptions~\ref{ass:Simplification} the initial condition $\Xini$ and disturbance $W_k$, $k\in\N^\infty$ admit exact PCEs in the basis $\Phi_\infty$. Similar to Proposition~\ref{pro:Basis}, we conclude that the overtakingly optimal solution to OCP~\eqref{eq:StochOCPInf} lives in the space spanned by the basis $\Phi_\infty$.

  	Consider the deterministic OCP~\eqref{eq:StochInfPCEj} for all $j\in\I_{[-2,\infty)}$. From the established results~\eqref{eq:SolutionDeterLQRInf}, one sees that $\pce{u}_k^{{-2},\sinf} = K\pce{x}_k^{{-2},\sinf} + F \mean[W]$ is the unique overtakingly optimal solution to OCP~\eqref{eq:StochInfPCEj} for $j=-2$. Especially, the solution is strongly optimal for the case $\mean[W]=0$, i.e., $J_\infty(\pce{x}\ini^{-2},\pce{u}^{-2,\sinf})<\infty$ and $J_\infty(\pce{x}\ini^{-2},\pce{u}^{-2,\sinf}) < J_\infty(\pce{x}\ini^{-2},\pce{u}^{-2})$ holds for any $\pce{u}^{-2}\neq \pce{u}^{-2,\sinf}$. Moreover, for all $j\in\I_{[-1,\infty)}$, $\pce{u}_k^{j,\sinf} = K\pce{x}_k^{j,\sinf}$ is the unique strongly optimal solution to OCP~\eqref{eq:StochInfPCEj} since the minimum cost is finite. Thus, it is also overtakingly optimal. Hence, the unique overtakingly optimal feedback for OCP~\eqref{eq:StochOCPInf} becomes $U_k^\sinf = (K\pce{x}_k^{-2,\sinf}{+}F \mean[W])\phi^{-2} + \sum_{j=-1}^{\infty} K\pce{x}_k^{j,\sinf}\phi^j = KX_k^\sinf + F \mean[W]$.
\end{proof}
Note that in Lemma~\ref{lem:Infinite} Assumption~\ref{ass:Simplification} is made only for the ease of notation and can be dropped as discussed in Section~\ref{sec:Extensions}.
	
Given the above optimal state feedback to OCP~\eqref{eq:StochOCPInf}, we first compute the optimal state trajectories in PCE coefficients for all $ j\in \I_{[-2,\infty)}$
\begin{equation*}
	\pce{x}_k^{j,\sinf} = \begin{cases}
		\tilde{A}^k \pce{x}\ini^{-2} {+} (I{-}\tilde{A})^{-1}(I{-}\tilde{A}^k)\tilde{F}\mean[W],  &\text{for } j=-2\\
		\tilde{A}^k\pce{x}\ini^{-1},  &\text{for } j=-1\\
			0,\hfill\text{for } k\leq j,\hspace*{16pt}&j\in\N^\infty\\
			\tilde{A}^{k-j-1} E\pce{w}^0,\hfill \text{for } k\geq j+1, &  j\in\N^\infty
	\end{cases}
\end{equation*}
with $\tilde{A}= A+BK$ and $\tilde{F} = BF+E$. The trajectory for $j=-2$ follows from $\pce{x}_k^{-2,\sinf} = \tilde{A}^k \pce{x}\ini^{-2} + \sum_{j=0}^{k-1}\tilde{A}^j\tilde{F}\mean[W]$. Recall that we assume $(A,B)$ stabilizable and $(A,Q^{1/2})$ detectable. Thus, $K$ is stabilizing and all eigenvalues of $\tilde{A}$ lie inside the unit circle \citep{anderson89optimal}. Therefore, for $k\to\infty$, $\pce{x}_k^{j,\sinf}$, $j\in \I_{[-2,\infty)}$ converge to their corresponding steady states exponentially with the same rate $\tilde{A}$, which is in line with the optimal finite-horizon trajectories sketched in Figure~\ref{fig:PCEIllustration}.

Moreover, for $j\in\N^\infty$ the PCE coefficients $\pce{x}_k^{j,\sinf}$ that are related to disturbances satisfy
\[
\pce{x}_k^{j,\sinf} = \pce{x}_{k-j}^{0,\sinf} = \pce{x}_{k+\tilde{k}}^{j+\tilde{k},\sinf},\quad \forall k\geq j+1,~\tilde{k}\in\N.
\]
That is, if the PCE coefficient dimension $j$ and the time step $k$ are increased simultaneously by the same value~$\tilde{k}$, the PCE coefficient is constant. Indeed this is a special case of Proposition~\ref{prop:PCETraj} for $\bar{A}_{k_1}^{k_2}=\tilde{A}^{k_2-k_1+1}$.

Since $X_k^\sinf =\sum_{j=-2}^{k-1}\pce{x}_k^{j,\sinf}\phi^{j}$, we can express the optimal state trajectory of OCP~\eqref{eq:StochOCPInf} in PCE coefficients
\begin{equation}  \label{eq:InfXk}
	\begin{split}
		X_k^\sinf =\Big( (I-\tilde{A})^{-1}(I-\tilde{A}^k)\tilde{F}\mean[W]+ \tilde{A}^k \pce{x}\ini^{-2}\Big)\phi^{-2} \\
		+ \tilde{A}^k\pce{x}\ini^{-1} \phi^{-1} + \textstyle{\sum_{j=0}^{k-1}} \tilde{A}^{k-j-1}E\pce{w}^0\phi^{j},~ k\geq 1.
	\end{split}
\end{equation}
The last item related to $\phi^{j}$, $j\in\I_{[0,k-1]}$ can be written as $\sum_{j=0}^{k-1} \tilde{A}^{k-j-1}E\pce{w}^0 \phi^{j} = \sum_{j=0}^{k-1} \tilde{A}^{k-j-1}E(W_j-\mean[W_j])$, which summarizes the accumulated influence of the past process disturbances. Recall that the PCE dimension $j$ coincides with the time step at which the disturbance $W_j$ enters the system. Due to the exponential decay of $\pce{x}_k^{j,\sinf}$, $j\in \I_{[0,\infty)}$, one can see that the most recent disturbances are dominant. In case of a finite optimization horizon, the quantification of the truncation error in Lemma~\ref{lem:Truncation} shows a similar behavior.

\subsection{Asymptotics of Optimal Trajectories}
First we recall the Wasserstein metric to quantify the distance between probability measures \citep{ruschendorf85wasserstein,villani09wasserstein}. Notice that in the following definition, $\|Z\|_2\coloneqq \sqrt{Z^\top Z}\in\lsp(\R)$ refers to the 2-norm on $R^{n_z}$ applied to $Z\in\lsp(\R^\dimz)$. That is, for any realization $Z(\omega)\in\R^{n_z}$, the 2-norm reads $\|Z(\omega)\|_2 = \sqrt{Z(\omega)^\top Z(\omega)}\in\R$.
\begin{defn}[Wasserstein metric] \label{def:WS}
	Consider two random variables $Z_1$, $Z_2\in\lsp(\R^\dimz)$ and $q\in[1,\infty]$. The Wasserstein distance of order~$q$ is
	\[
	\mcl{W}_q(Z_1,Z_2) \coloneqq {\inf_{\tilde{Z}_1,\tilde{Z}_2}}\Big(\mean\big[\|\tilde{Z_1}-\tilde{Z_2}\|_2^q\big]^{\frac1q}, \tilde{Z}_1{\sim} Z_1,\tilde{Z}_2{\sim} Z_2\Big),
	\]
	where $\sim$ denotes the equivalence in distribution, i.e., $\tilde{Z}_t$ and $Z_t$, $t\in\{1,2\}$ follow the same distribution.
\end{defn}
With slight abuse of notation, the Wasserstein distance between the measures $\mu_1$, $\mu_2$ is $\mcl{W}_q(\mu_1,\mu_2)\coloneqq\mcl{W}_q(Z_1,Z_2)$ with $\mu_{Z_t}=\mu_t$ for $t\in\{1,2\}$, where $\mu_{Z_t}$ denotes the push-forward measure $\mu_{Z_t}(\cdot)\coloneqq \mu(Z_t^{-1}(\cdot))$. Moreover, two measures or random variables are said to be equivalent in the Wasserstein metric if
\[
\mu_1 \ed \mu_2 \quad \iff \quad \mcl{W}_q(\mu_1,\mu_2)=0.
\]
Note that the order $q$ does not play a role in the equivalence if $\mcl{W}_q(\mu_1,\mu_2)$ exists. Thus $q$ is henceforth omitted.

For the ease of notation, we use the shorthand $(X^\sinf,U^\sinf)$ to denote the optimal trajectory $\{(X_k^\sinf,U_k^\sinf)\}_{k=0}^\infty$ of OCP~\eqref{eq:StochOCPInf} as the superscript $\cdot^\sinf$ refers to infinite horizon. Additionally, the first two moments and cost function of a pair of probability measures $(\mu_X,\mu_U)$ are defined via the corresponding state-input pair. That is,  $\mean  [\mu_X]\coloneqq \mean[X]$, $\Sigma [\mu_X] \coloneqq \Sigma [X]$, and $\ell(\mu_X,\mu_U)= \ell(X,U)$ follow for any $(X,U)\ed (\mu_X,\mu_U)$.
\begin{defn}[Stationary pair]
	$(\bar{X},\bar{U})$ is said to be a stationary pair of system~\eqref{eq:Sys} if $\bar{X} \ed A\bar{X}+B\bar{U}+EW$ holds, where $W$ is the process disturbance independent of $\bar{X}$ and $\bar{U}$. Moreover, $\{(\bar{X}_k,\bar{U}_k)\}_{k=0}^N$, $N\in\N^\infty$ is a stationary trajectory if $(\bar{X}_{k+1},\bar{U}_{k+1})\ed (\bar{X}_k,\bar{U}_k)$, $\forall k\in\I_{[0,N-1]}$ holds.
\end{defn}
In general, as the disturbance $W$ is independent of $\bar{X}$, $\bar{X} = A\bar{X}+B\bar{U}+EW$ does not hold. However, the next lemma gives an explicit expression of a stationary pair in the sense of Wasserstein metric.
\begin{lem}[Infinite-horizon asymptotics] \label{lem:Limit}
	The optimal trajectory $(X^\sinf,U^\sinf)$ of OCP~\eqref{eq:StochOCPInf} converges in probability measure to
	\begin{subequations} \label{eq:InfXU}
		\begin{align}	
			&\hspace{5pt}(\mu_X^\sinf,\mu_U^\sinf)\coloneqq \big(\mu_{X_\infty^{\sinf}},\mu_{U_\infty^{\sinf}}\big) = \lim_{k\to\infty} \big(\mu_{X_k^{\sinf}},\mu_{U_k^{\sinf}}\big),\\	
			&\mu_X^\sinf   \ed (I{-}\tilde{A})^{-1}\tilde{F}\mean[W] {+} \textstyle{\sum_{j=0}^{\infty}} \tilde{A}^{j}E(W_j{-}\mean[W]),\\
			&\mu_U^\sinf \ed K\left((I{-}\tilde{A})^{-1}\tilde{F}{+} F\right) \mean[W] \nonumber\\
			&\hspace{80pt} + K\textstyle{\sum_{j=0}^{\infty}} \tilde{A}^{j}E(W_j{-}\mean[W]).
		\end{align}
	\end{subequations}
	The first two moments of $\mu_X^\sinf$ are
	\begin{subequations} \label{eq:MomentsX}
		\begin{align}
			\mean  [\mu_X^\sinf] &= (I_{n_x}-\tilde{A})^{-1}\tilde{F}\mean[W],\\
			\Sigma [\mu_X^\sinf]  & = \textstyle{\sum_{j=0}^{\infty}} \tilde{A}^{j}E\Sigma[W] E^\top \tilde{A}^{j\top}. \label{eq:MomentsXCovariance}
		\end{align}
	\end{subequations}
	Additionally, any $(X,U)$ satisfying $X\ed\mu_X^\sinf$ and $U=KX+F\mean[W]$ is a stationary pair of system~\eqref{eq:Sys}.
\end{lem}
\begin{proof}
	To simplify the notation, we first consider Assumption~\ref{ass:Simplification} to hold.
	The PCE expression of $X_k^\sinf$ in \eqref{eq:InfXk} gives
	\begin{align*}
		X_\infty^\sinf &\ed \lim_{k\to\infty} \big( (I{-}\tilde{A})^{-1}(I{-}\tilde{A}^k)\tilde{F}\mean[W]{+} \tilde{A}^k \pce{x}\ini^{-2}  \big)\phi^{-2}\\
		& \hspace{45pt} {+}\lim_{k\to\infty} \tilde{A}^k\pce{x}\ini^{-1} \phi^{-1} {+} \textstyle{\sum_{j=0}^{\infty}} \tilde{A}^{j}E\pce{w}^0\phi^j \\
		&= (I{-}\tilde{A})^{-1}\tilde{F}\mean[W] \phi^{-2} + \textstyle{\sum_{j=0}^{\infty}} \tilde{A}^{j}E\pce{w}^0 \phi^j,\\
		U_\infty^\sinf &=  F \mean[W] + KX_\infty^\sinf.
	\end{align*}
	Then it is straightforward to obtain the first two moments of $X_\infty^\sinf$ as~\eqref{eq:MomentsX}. Note that $\Sigma [X_\infty^\sinf]$ is the unique positive semidefinite solution to the discrete-time Lyapunov equation $\tilde{A}\Sigma [X_\infty^\sinf]\tilde{A}^\top -\Sigma [X_\infty^\sinf] + E\Sigma[W]E^\top=0$, where $\Sigma[W] \succeq 0$ \citep{simoncini16computational}. Since $U^\sinf=KX^\sinf+F\mean[W]$, $\mean  [Z_\infty^\sinf]<\infty$ and $\Sigma [Z_\infty^\sinf]<\infty$ follow for $Z\in\{X,U\}$. Hence $(X_\infty^\sinf,U_\infty^\sinf)$ exists and lives in an $\lsp$ space, i.e., $X_k^\sinf\in\splx{n_x}$ and $U_k^\sinf\in\splx{n_u}$ for $k\in\N^\infty$.
	
	Next we prove that $(X_\infty^\sinf,U_\infty^\sinf)$ is a stationary pair. Let $W=\mean[W]\phi^{-2}+ \pce{w}^0\phi^w$ be the disturbance that is independent of $X_\infty^\sinf$ and $U_\infty^\sinf$ with $\phi^w \ed \phi^j$ for all $j\in\N^\infty$. Then we have
	\begin{align*}
		& AX_\infty^\sinf {+} BU_\infty^\sinf {+} EW = \tilde{A}X_\infty^\sinf {+} BF\mean[W] {+}  EW\\
		\ed&\tilde{A}\Big((I{-}\tilde{A})^{-1}\tilde{F}\mean[W] \phi^{-2} {+} \textstyle{\sum_{j=0}^{\infty}} \tilde{A}^{j}E\pce{w}^0\phi^j\Big) \\
		&\hspace{80pt} +BF\mean[W] {+} E(\mean[W]\phi^{-2}{+} \pce{w}^0\phi^w)\\
		=&\Big(\tilde{A}(I{-}\tilde{A})^{-1}{+}I\Big) \tilde{F}\mean[W] \phi^{-2} {+}  E\pce{w}^0\phi^w{+} \textstyle{\sum_{j=1}^{\infty}} \tilde{A}^{j}E\pce{w}^0\phi^j,\\
		\ed& (I{-}\tilde{A})^{-1}\tilde{F}\mean[W] \phi^{-2} + \textstyle{\sum_{j=0}^{\infty}} \tilde{A}^{j}E\pce{w}^0\phi^j \ed X_\infty^\sinf,
	\end{align*}
	where $\phi^{-2}=1$ and $\pce{w}^0\phi^j = W_j-\mean[W]$ for $j\in\N^\infty$. Hence, any $(X,U)$ satisfying $X\ed X_\infty^\sinf$ and $U= KX+F\mean[W]$ is a stationary pair.
	
	To drop Assumption~\ref{ass:Simplification}, we replace $\pce{x}_{\ini}^{-1}\phi^{-1}$ and $w^0\phi^j$ in the above proof with $X_{\ini}-\mean[X_{\ini}]$ and $W_j-\mean[W]$, respectively. The corresponding decomposition of $W$ reads $W=\mean[W] + \big(W-\mean[W])$. Then following along the same line, the above proof still holds without Assumption~\ref{ass:Simplification}. This way, we relax Assumption~\ref{ass:Simplification}. As we have shown that $(X_\infty^\sinf,U_\infty^\sinf)$ lives in an $\lsp$ probability space and is a stationary pair, the convergence of $(X^\sinf,U^\sinf)$ in probability measure follows.
\end{proof}

\subsection{Convergence rate}
Given any trajectory of a stochastic LTI system, the concept of corresponding stationary trajectory is introduced by \citet{schiessl23pathwise}.
\begin{defn}[Stationary trajectory]
	Given any state-input-disturbance trajectory $\{(X_k,U_k,W_k)\}_{k=0}^N$, $N\in\N^\infty$ of system \eqref{eq:Sys}, $\{(\bar{X}_k, \bar{U}_k)\}_{k=0}^N$ satisfying
	\begin{align*}
		\begin{bmatrix} X_{k+1} \\ \bar{X}_{k+1} \end{bmatrix} &= \begin{bmatrix} A & 0 \\ 0 & A \end{bmatrix} \begin{bmatrix} X_k \\ \bar{X}_k \end{bmatrix} + \begin{bmatrix} B & 0 \\ 0 & B \end{bmatrix} \begin{bmatrix} U_k \\ \bar{U}_k \end{bmatrix} + \begin{bmatrix} E \\ E \end{bmatrix}W_k,\\
		\begin{bmatrix} X_0 \\ \bar{X}_0 \end{bmatrix} &= \begin{bmatrix} X\ini \\ \bar{X}\ini \end{bmatrix},\quad (\bar{X}_{k+1},\bar{U}_{k+1})\ed (\bar{X}_k,\bar{U}_k).
	\end{align*}
	is called corresponding stationary trajectory.
	Additionally, $\{(\bar{X}_k^\sinf,\bar{U}_k^\sinf)\}_{k=0}^N$ denotes the corresponding stationary trajectory with probability measure  $(\mu_X^\sinf, \mu_U^\sinf)$, i.e. $(\bar{X}_k^\sinf,\bar{U}_k^\sinf)\ed (\mu_X^\sinf, \mu_U^\sinf)$.
\end{defn}
With the above definition, we obtain the convergence of $(\bar{X}_k^\sinf,\bar{U}_k^\sinf)$ to its corresponding stationary trajectory.
\begin{lem}[Exponential convergence of $(\bar{X}_k^\sinf,\bar{U}_k^\sinf)$] \label{lemma:Convergence}
	Let $(X^\sinf,U^\sinf)$ be the optimal state-input trajectory of OCP~\eqref{eq:StochOCPInf}, and let $(\bar{X}^\sinf,\bar{U}^\sinf)$ be the corresponding stationary trajectory with measure $(\mu_X^\sinf, \mu_U^\sinf)$. Then there exist constants $\beta>0$, $p \in [0,1)$, such that
	\[
	\|(X_k^\sinf,U_k^\sinf) -(\bar{X}_k^\sinf,\bar{U}_k^\sinf)\| \leq \beta p^k.
	\]
	That is, $(X^\sinf,U^\sinf)$ converges to $(\bar{X}^\sinf,\bar{U}^\sinf)$ exponentially in the sense of the $\mcl{L}^2$-norm. Consequently, $(X^\sinf,U^\sinf)$ converges almost surely to $(\bar{X}^\sinf,\bar{U}^\sinf)$, i.e., for any $\varepsilon>0$, we have
	\begin{equation*}
		\mathbb{P}\Big(\limsup_{k \rightarrow \infty} \{ \| (X^\sinf_k,U^\sinf_k) - (\bar{X}^\sinf_k,\bar{U}^\sinf_k) \|_{2} \geq \varepsilon \} \Big) = 0.
	\end{equation*}
\end{lem}
\begin{proof}
	The input feedback policies $U_k^\sinf = K X_k^\sinf + F\mean[W]$ and $\bar{U}_k^\sinf = K \bar{X}_k^\sinf + F\mean[W]$ imply that $X_{k+1}^\sinf - \bar{X}_{k+1}^\sinf = \tilde{A} (X_k^\sinf - \bar{X}_k^\sinf)$.
	It follows that $\|(X_k^\sinf - \bar{X}_{k}^\sinf)\| \leq \|\tilde{A}^k\| \|X_0^\sinf - \bar{X}_0^\sinf\|$. Let the eigenvalue decomposition of $\tilde{A}$ be $\tilde{A}=V\Lambda V^{-1}$ with diagonal matrix $\Lambda$ and $\rho_{\tilde{A}}$ be the largest eigenvalue of $\tilde{A}$. Then we get $\|\tilde{A}^k\|\leq \|V\|_2 \|\Lambda^k\|\|V^{-1}\|_2 \leq \rho_{\tilde{A}}^k\|V\|_2 \|V^{-1}\|_2$, where $\|V\|_2$ denotes the 2-norm of the matrix $V$. Since $(A,B)$ is stabilizable and $(A,Q^{1/2})$ is detectable, all the eigenvalues of $\tilde{A}$ are inside the unit circle and thus  $0\leq \rho_{\tilde{A}}<1$. Therefore, we obtain $\|X_k^\sinf - \bar{X}_k^\sinf\| \leq \beta_x \rho_{\tilde{A}}^k$ with $\beta_x = \|V\|_2\cdot\|V^{-1}\|_2\cdot \|X_0^\sinf - \bar{X}_0^\sinf\|$. 	Moreover, the optimal input policy $U_k^\sinf =KX_k^\sinf + F \mean[W]$ suggests $\|U_k^\sinf {-} \bar{U}_k^\sinf\| \leq \|K\|_2 \beta_x \rho_{\tilde{A}}^k$.
	Finally, we obtain $\|(X_k^\sinf,U_k^\sinf) -(\bar{X}_k^\sinf,\bar{U}_k^\sinf)\| \leq \sqrt{1+ \|K\|_2^2} \beta_x \rho_{\tilde{A}}^k = \beta p^k$ with $\beta\coloneqq \sqrt{1+ \|K\|_2^2} \beta_x$ and $p\coloneqq\rho_{\tilde{A}}$.
	
	By Markov's inequality the exponential convergence implies that $\sum_{k=0}^{\infty} \mathbb{P}(\| (X^\sinf_k,U^\sinf_k) - (\bar{X}^\sinf_k,\bar{U}^\sinf_k) \|_{2} \geq \varepsilon ) \leq  \sum_{k=0}^{\infty} \|(X_k^\sinf,U_k^\sinf) -(\bar{X}_k^\sinf,\bar{U}_k^\sinf)\|/\varepsilon \leq \sum_{k=0}^{\infty} \beta p^k/\varepsilon = \beta/\big(\varepsilon (1-p)\big) < \infty$ holds for all $\varepsilon > 0$ \citep{bertsekas08introduction}. Using the Borel-Cantelli lemma almost sure convergence follows \citep{feller71introduction}.
\end{proof}
\section{Optimal Stationary Solutions} \label{sec:Stationary}
In this section, we give the unique optimal solution to the stationary optimization problem in closed form. Additionally, we provide the finite-dimensional approximation of the infinite-dimensional optimal solution for a given error bound.
\subsection{Optimal Stationary Pair}
The deterministic Stationary Optimization Problem (SOP) that corresponds to the deterministic linear--quadratic OCP \eqref{eq:DeterOCP}  is given by
\begin{equation*}
		\min_{\bar{x}\in\R^{\dimx}, \bar{u}\in\R^{\dimu}}~ \ell(\bar{x},\bar{u})\quad \text{s.t.} \quad \bar{x} =  A\bar{x}+B\bar{u} +Ec.
\end{equation*}
The optimal stationary pair is denoted by $(\bar{x}^\star,\bar{u}^\star)$. The stochastic counterpart to the SOP reads
\begin{equation} \label{eq:SOP}
		\min_{\substack{\bar{X}\in\lsp(\R^\dimx),\\ \bar{U}\in\lsp(\R^\dimu)}}~ \ell(\bar{X},\bar{U})\quad \text{s.t.} ~ \bar{X} \ed  A\bar{X}+B\bar{U}+EW 
\end{equation}
whose optimal solutions are denoted as $(\bar{X}^\star,\bar{U}^\star)$.

Consider an infinite-dimensional basis $\{\bar{\phi}^j\}_{j=0}^\infty$ which spans the $\dimx$ spatial dimensions of $\lsp(\R^\dimx)$. We have $\bar{Z} =\sum_{j\in\N^\infty} \bar{\pce{z}}^j\bar{\phi}^j$, $(Z,z)\in\{(U,u),(X,x)\}$. Similar to~\eqref{eq:StochPCE}, by replacing all random variables in~\eqref{eq:SOP} with their PCEs, the reformulated~\eqref{eq:SOP} follows
\begin{equation} \label{eq:SOPPCE}
	\begin{split}
			&\min_{\bar{\pce{u}}^j \in\R^{n_u}, \bar{\pce{x}}^j\in\R^{n_x},j\in\N^\infty}~  \sum_{j\in\N^\infty} \ell(\bar{\pce{x}}^j,\bar{\pce{u}}^j)\| \bar{\phi}^j\|^2\\
		\text{s.t.}& ,\quad \bar{\pce{x}}_+^{j} = A\bar{\pce{x}}^j + B\bar{\pce{u}}^j + E\pce{w}^j,\quad j\in \N^\infty,\\
		&\quad \sum_{j\in\N^\infty}\bar{\pce{x}}_+^j\bar{\phi}^j \ed \sum_{j\in\N^\infty}\bar{\pce{x}}^j\bar{\phi}^j.
	\end{split}
\end{equation}
SOP~\eqref{eq:SOP} differs from its deterministic counterpart as follows:
(i) The problem is not directly tractable as the decision variables $\bar{X}$, $\bar{U}$ are random variables. (ii) Though the solution to SOP~\eqref{eq:SOP} is unique in the sense of measure, \eqref{eq:SOP} admits infinitely many solutions in terms of random variables with the same measure. (iii)~The PCE reformulated problem~\eqref{eq:SOPPCE} is also difficult to handle due to the constraint on the equivalence in the Wasserstein metric. Since each solution in random variables corresponds to a solution in PCE coefficients, the PCE reformulated OCP~\eqref{eq:SOPPCE} also admits infinitely many solutions.

Therefore, the probability measure of the solutions to SOP~\eqref{eq:SOP} is of interest and is denoted by $(\bar{\mu}_X^\star, \bar{\mu}_U^\star)$. The next results show that $(\bar{\mu}_X^\star, \bar{\mu}_U^\star) \ed (\mu_X^\sinf, \mu_U^\sinf)$, which establishes the link between SOP~\eqref{eq:SOP} and infinite-horizon OCP~\eqref{eq:StochOCPInf}.
\begin{thm}[Unique optimal stationary pair] \label{theorem:SOP}
	Let $(\bar{\mu}_X^\star, \bar{\mu}_U^\star)$ be the probability measure of the solutions to \eqref{eq:SOP}. Then $(\bar{\mu}_X^\star, \bar{\mu}_U^\star)$ is uniquely determined by $(\mu_X^\sinf, \mu_U^\sinf)$, i.e. $(\bar{\mu}_X^\star, \bar{\mu}_U^\star) \ed (\mu_X^\sinf, \mu_U^\sinf)$. Moreover, the minimum cost is $	\ell(\bar{\mu}_X^\star, \bar{\mu}_U^\star) = \|W\|_{E^\top PE}^2 +  \|\mean[W]\|_{\Delta \bar{S}}^2$,
where $\Delta \bar{S} = E^\top G + G^\top E - F^\top (R+B^\top P B)F$.
\end{thm}
\begin{proof}
Consider infinite-horizon OCP~\eqref{eq:StochOCPInf} and let $W$ follow the Dirac distribution, i.e., $W$ is deterministic and thus $W=\mean[W]$. Lemma~\ref{lem:Infinite} implies that the considered system~\eqref{eq:Sys} is optimally operated at steady state $(\mean[\mu_X^\sinf],\mean[\mu_U^\sinf])$. Thus, $(\mean[\mu_X^\sinf],\mean[\mu_U^\sinf])$ is an optimal solution to the deterministic counterpart of SOP SOP~\eqref{eq:SOP}. Therefore, the assumptions of Theorem~5.2 by \cite{schiessl24turnpike} are satisfied and we can show that $(\mu_X^\sinf, \mu_U^\sinf)$ is an optimal solution to SOP~\eqref{eq:SOP} with any $\lsp$ process disturbance in the sense of probability measure. Moreover, this indicates that $\bar{\mu}_X^\star$ is uniquely determined by $\mu_X^\sinf$. Then the minimum cost immediately follows as the measure $(\bar{\mu}_X^\star, \bar{\mu}_U^\star)$ is known.

Next, we show by contradiction that $\bar{\mu}_U^\star$ is also unique. Assume there exists another stationary pair $(\hat{X},\hat{U})$ minimizing SOP~\eqref{eq:SOP} with $\hat{X}\ed\mu_X^\sinf$ and $\mathcal{W}_q(\hat{U}_k,\mu_U^\sinf)\neq 0$. That is, $\ell(\hat{X},\hat{U})= \ell(\mu_X^\sinf, \mu_U^\sinf)$. Consider infinite-horizon OCP~\eqref{eq:StochOCPInf} with initial condition $X\ini = \hat{X}$. Then the trajectory $\{(\hat{X}_k,\hat{U}_k)\}_{k=0}^\infty$ with $\hat{X}_0=\hat{X}$ and $(\hat{X}_k,\hat{U}_k)\ed(\hat{X},\hat{U})$, $k\in\N^\infty$ is an overtakingly optimal solution to OCP~\eqref{eq:StochOCPInf}. As Lemma~\ref{lem:Infinite} states that $\hat{U}_k =K\hat{X}_k + F \mean[W]\ed\mu_U^\sinf$ is the unique overtakingly optimal solution to OCP~\eqref{eq:StochOCPInf}, we arrive at a contradiction.
\end{proof}

Compared to Theorem~5.2 by \citet{schiessl24turnpike}, Theorem~\ref{theorem:SOP} offers an avenue to obtain the analytical stationary solution in closed form via the solution to the corresponding infinite-horizon OCP~\eqref{eq:StochOCPInf}. Note that Theorem~\ref{theorem:SOP} applies in a quite general setting beyond Gaussianity.

\subsection{Finite-dimensional Approximation}
As $(\bar{\mu}_X^\star, \bar{\mu}_U^\star)$ is of infinite dimension, one needs to truncate the series $\textstyle{\sum_{j=0}^{\infty}} \tilde{A}^{j}E(W_j{-}\mean[W])$, which causes a corresponding truncation error.
The following lemma points towards the computation via truncated series and thus gives a numerically tractable approximation of $(\bar{\mu}_X^\star,\bar{\mu}_U^\star)$. We recall the eigenvalue decomposition $\tilde{A} = V\Lambda V^{-1}$ and that $\rho_{\tilde{A}}$ is the largest eigenvalue of $\tilde{A}$ with $\rho_{\tilde{A}}<1$.

\begin{lem}[Finite-dimensional approximation] \label{lem:FiniteApprox}
	Consider the SOP~\eqref{eq:SOP} and its optimal solutions $(\bar{X}^\star,\bar{U}^\star)$. Let the approximation of $(\bar{X}^\star,\bar{U}^\star)$ be
	\begin{align*}
		\bar{X}^{\text{trun},\star}(p+1)  &\coloneqq (I{-}\tilde{A})^{-1}\tilde{F}\mean[W]{+}\sum_{j=0}^{p-1} \tilde{A}^{j}E(W_j{-}\mean[W]),\\
		\bar{U}^{\text{trun},\star}(p+1)  &\coloneqq K\bar{X}^{trun,\star}(p+1)+F\mean[W].
	\end{align*}
    For a user-chosen error bound $\delta>0$, we define $\tilde{p}(\delta):\R^{+}\to\N$ as
		\begin{equation} \label{eq:LeastDimension}
		\tilde{p}(\delta) = \lceil \big(\ln(\delta)+c\big)/\ln(\rho_{\tilde{A}}) \rceil,
	\end{equation}
	where $\lceil\cdot\rceil$ denotes the ceiling function and $ c = \ln(1-\rho_{\tilde{A}})-\ln \big(\sqrt{1+\|K\|_2^2}\tr(\Sigma[W]E^\top E)\|V\|_2\|V^{-1}\|_2\big)$. Then for all $p \geq \tilde{p}(\delta)$, $p\in\N$, we have
		\begin{equation} \label{eq:ErrorBound}
			\mcl{W}_2\Big((\bar{X}^\star,\bar{U}^\star), \big(\bar{X}^{\text{trun},\star}(p+1), \bar{U}^{\text{trun},\star}(p+1)\big)\Big) \leq\delta.
		\end{equation}
\end{lem}
\begin{proof}
	As Theorem~\ref{theorem:SOP} has shown that $(\mu_X^\sinf,\mu_U^\sinf)$ is the probability measure of $(\bar{X}^\star,\bar{U}^\star)$, the above $(p+1)$-dimensional approximation immediately follows from~\eqref{eq:InfXU}, while the truncation error is
	\begin{align*}
		\Delta \bar{X}^\star(p+1) &\coloneqq \bar{X}^\star{-}\bar{X}^{\text{trun},\star} = \sum_{j=p}^{\infty} \tilde{A}^{j}E\left(W_j{-}\mean[W]\right),\\
		\Delta \bar{U}^\star(p+1) &\coloneqq \bar{U}^\star{-}\bar{U}^{\text{trun},\star} = K\sum_{j=p}^{\infty} \tilde{A}^{j}E\left(W_j{-}\mean[W]\right).
	\end{align*}
	Then from the eigenvalue decomposition $\tilde{A}=V\Lambda V^{-1}$ and the definition of the Wasserstein metric we obtain
	\begin{align*}
		&\mcl{W}_2(\bar{X}^{\text{trun},\star}(p+1),\bar{X}^\star) \leq \| \Delta \bar{X}^\star(p+1) \|\\
		= &\sqrt{\mean\Big[\big(W_j{-}\mean[W]\big)^\top E^\top\Big(\sum_{j=p}^\infty\tilde{A}^{j\top}\tilde{A}^{j}\Big) E\big(W_j{-}\mean[W]\big)\Big]}\\
		\leq & \|V\|_2 \cdot \|V^{-1}\|_2\cdot\big\| E\big(W-\mean[W]\big) \big\|  \sum_{j=p}^{\infty}\rho_{\tilde{A}}^j\\
		= &\|V\|_2 \cdot \|V^{-1}\|_2\cdot\tr\big(\Sigma[W]E^\top E\big)\cdot \rho_{\tilde{A}}^p(1-\rho_{\tilde{A}})^{-1}.
	\end{align*}
	Since $\mcl{W}_2(\bar{U}^{\text{trun},\star}(p+1),\bar{U}^\star) = \|K\|_2\mcl{W}_2(\bar{X}^{\text{trun},\star}(p+1),\bar{X}^\star)$, we have
	\begin{align*}
		&\mcl{W}_2\Big((\bar{X}^\star,\bar{U}^\star), \big(\bar{X}^{\text{trun},\star}(p+1), \bar{U}^{\text{trun},\star}(p+1)\big)\Big)\\
		\leq &  \sqrt{1+\|K\|_2^2}\cdot\|V\|_2  \|V^{-1}\|_2\tr\big(\Sigma[W]E^\top E\big) \rho_{\tilde{A}}^p(1-\rho_{\tilde{A}})^{-1}.
	\end{align*}
	Letting $ \sqrt{1+\|K\|_2^2}\|V\|_2  \|V^{-1}\|_2\tr\big(\Sigma[W]E^\top E\big) \rho_{\tilde{A}}^p(1-\rho_{\tilde{A}})^{-1} \leq \delta$, \eqref{eq:LeastDimension} immediately follows.
\end{proof}
Given an arbitrary error bound $\delta >0$ for the Wasserstein metric, Lemma~\ref{lem:FiniteApprox} determines the dimensions of the approximations $\left(\bar{X}^{\text{trun},\star}(p+1), \bar{U}^{\text{trun},\star}(p+1)\right)$ for the error bound to hold. Note that~\eqref{eq:LeastDimension} is only a sufficient condition for~\eqref{eq:ErrorBound}. That is, there may exist $p <\tilde{p}(\delta)$ such that \eqref{eq:ErrorBound} also holds. Indeed, the proof of Lemma~\ref{lem:FiniteApprox} has shown that the approximation error of a $(p+1)$-dimensional approximation satisfies 
\begin{align*}
	&\mcl{W}_2\Big((\bar{X}^\star,\bar{U}^\star), \big(\bar{X}^{\text{trun},\star}(p+1), \bar{U}^{\text{trun},\star}(p+1)\big)\Big)\\
	\leq&\sqrt{1+\|K\|_2^2}\cdot\|\Delta \bar{X}^\star(p+1)\|\\
	=&\sqrt{1+\|K\|_2^2} \cdot \big\|E\big(W{-}\mean[W]\big)\big\|_{M(p)},
\end{align*}
where $M(p)$ is the unique positive semidefinite solution to the Lyapunov equation $\tilde{A}^\top M(p)\tilde{A} - M(p) + (\tilde{A}^p)^\top \tilde{A}^{p}=0$. Comparing the Lyapunov functions of arbitrary $p\in\N$ and $p+1$, we get $M(p+1)=\tilde{A}^\top M(p)\tilde{A}$ and thus $M(p)=\tilde{A}^{p\top}M(0)\tilde{A}^p$ follows. From the eigenvalue decomposition $\tilde{A}=V\Lambda V^{-1}$, the exponential convergence follows as $\|M(p)\|_2\leq\|V\|_2^2 \cdot\|V^{-1}\|_2^2\cdot\|M_0\|_2\rho_{\tilde{A}}^p$.
Hence, to satisfy an arbitrary required error bound $\delta>0$, we can solve the above Lyapunov equation repeatedly for $p=0,1,2,...,$ up to $\bar{p}$ such that the inequality
\begin{equation} \label{eq:LeastDimensionAlternative}
	\sqrt{1+\|K\|_2^2} \cdot \big\|E\big(W{-}\mean[W]\big)\big\|_{M(\bar{p})}\leq\delta.
\end{equation}
holds. Then \eqref{eq:ErrorBound} holds for all $p\geq\bar{p}(\delta)$. This way, we obtain another approximation $\big(\bar{X}^{\text{trun},\star}(\bar{p}+1), \bar{U}^{\text{trun},\star}(\bar{p}+1)\big)$ for the error bound $\delta$, though we cannot explicitly express $\bar{p}(\delta)$ in closed form. Since $\bar{p}(\delta)$ and $\tilde{p}(\delta)$ both satisfy the inequality~\eqref{eq:LeastDimensionAlternative}, while $\bar{p}$ is the minimum integer to have \eqref{eq:LeastDimension} hold, $\bar{p}(\delta)\leq\tilde{p}(\delta)$ follows.

While $(\bar{\mu}_X^\star, \bar{\mu}_U^\star)$ and $\left(\bar{X}^{\text{trun},\star}(p+1), \bar{U}^{\text{trun},\star}(p+1)\right)$ derived in Lemma~\ref{lem:FiniteApprox} are not expressed via PCE, they can be computed this way.
To this end, consider a $(p+1)$-dimensional approximation in Lemma~\ref{lem:FiniteApprox} and a corresponding joint basis~$\Phi=\{\phi^j\}_{j=0}^{L-1}$ constructed as \eqref{eq:BasisElement} with $L=1+p(L_w-1)$, $L_w\in\N^\infty$. The PCE of $W_k$, $k\in\I_{[0,p-1]}$ in the joint basis $\Phi$ is
\[
	W_k = \sum_{j\in\{0\}\cup\mcl{I}_k}\pce{w}_k^j\phi^j~\text{with}~ \mcl{I}_k=\I_{[k(L_w-1)+1, (k+1)(L_w-1)]},
\]
since $\pce{w}_k^j=0$ for all $j\in\I_{[1,L-1]}\setminus\mcl{I}_k$. Then we rewrite $\left(\bar{X}^{\text{trun},\star}(p+1), \bar{U}^{\text{trun},\star}(p+1)\right)$ via PCE as
\begin{equation}\label{eq:ApproximationPCE}
	\bar{X}^{\text{trun},\star}  = (I{-}\tilde{A})^{-1}\tilde{F}\mean[W]{+}\sum_{k=0}^{p-1}\sum_{j\in\mcl{I}_k} \tilde{A}^{j}E\pce{w}_k^j\phi^j
\end{equation}
and $\bar{U}^{\text{trun},\star} = K\bar{X}^{trun,\star}+F\mean[W]$.
This way, one can  compute $\left(\bar{X}^{\text{trun},\star}(p+1), \bar{U}^{\text{trun},\star}(p+1)\right)$ numerically in the PCE framework.
Additionally, the first two moments of $\Delta \bar{X}^\star(p+1)$ are $\mean\left[\Delta \bar{X}^\star(p+1)\right] =0$ and $\Sigma\left[\Delta \bar{X}^\star(p+1)\right] = \textstyle{\sum_{j=p}^{\infty}} \tilde{A}^{j}E\Sigma[W] E^\top \tilde{A}^{j\top}$. 
Therefore, given an arbitrary error bound in the sense of the first two moments, the required PCE dimension of the approximation can be derived in a similar fashion as in Lemma~\ref{lem:FiniteApprox}.
\section{Numerical Example} \label{sec:Simulation}
We modify the linearized CSTR reactor from \citet{faulwasser18asymptotic}. As the original system is stable, we modify the dynamics as
\[
X_{k+1} = \begin{bmatrix} 1.24 & 0\phantom{.0} \\ 0.12 & 0.2 \end{bmatrix} X_k + \begin{bmatrix} -0.5 \\ \phantom{-}0.5 \end{bmatrix} U_k 
+ \begin{bmatrix} 1 \\ 1 \end{bmatrix}W_k,
\]
where $W_k$, $k\in\N$ are modeled as i.i.d. scalar random variables that follow a uniform distribution with support $[0,~0.6]$. The initial condition is $X\ini = [0.4,~1.5]^\top +[0.4,~1.0]^\top\theta$, where $\theta\sim\mathcal{N}(0,1)$ is a standard Gaussian random variable. Then we have $L\ini=L_w=2$ and Assumption~\ref{ass:Simplification} is thus satisfied. The weighting matrices are $Q=\text{diag}([1,1])$, $R=1$, and $Q_N=P=\big[\begin{smallmatrix} 5.31\phantom{0} & 0.177 \\ 0.177 & 1.04\phantom{0} \end{smallmatrix}\big]$, where $P$ is the stationary solution to the discrete-time algebraic Riccati equation~\eqref{eq:RiccatiP}. Note that $(A,B)$ is controllable, and $(A,Q^{1/2})$ is detectable.

\begin{figure}[t]
	\begin{subfigure}{\linewidth}
		\centering
		\includegraphics[width=0.75\linewidth,trim={65mm 28mm 40mm 36mm},clip]{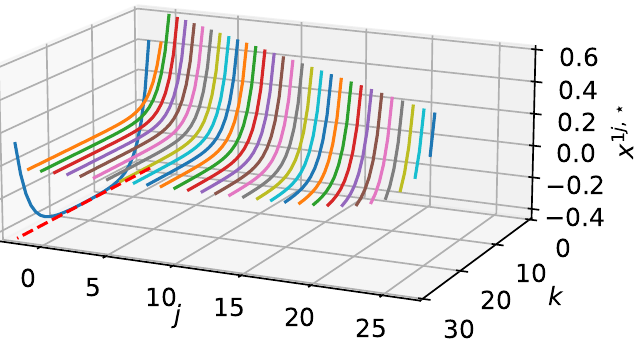}
	\end{subfigure}
	\begin{subfigure}{\linewidth}
		\centering
		\includegraphics[width=0.75\linewidth,trim={65mm 28mm 40mm 36mm},clip]{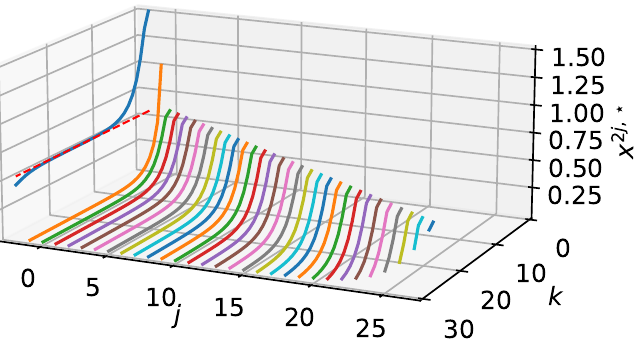}
	\end{subfigure}\\[1ex]
	\begin{subfigure}{\linewidth}
		\centering
		\includegraphics[width=0.75\linewidth,trim={65mm 28mm 40mm 36mm},clip]{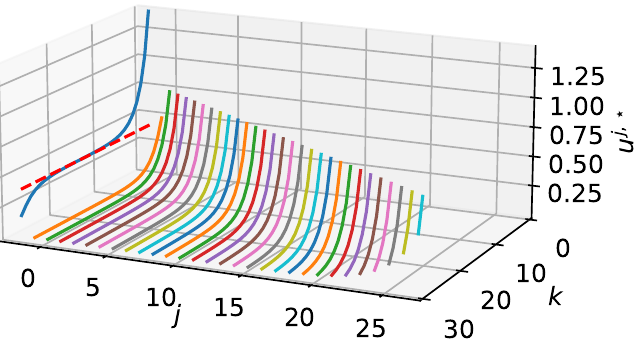}
	\end{subfigure}
	\caption{Trajectories of PCE coefficients  $\{(\pce{x}_k^{j,\star},\pce{u}_k^{j,\star})\}_{k=0}^{29}$, $j\in\I_{[-2,29]}$. Red-dashed line: Expectation of $(\mu_X^\sinf,\mu_U^\sinf)$.}
	\label{fig:CSTRPCEs}
\end{figure}

\begin{figure}[t]
	\begin{center}
		\includegraphics[width=0.5\textwidth]{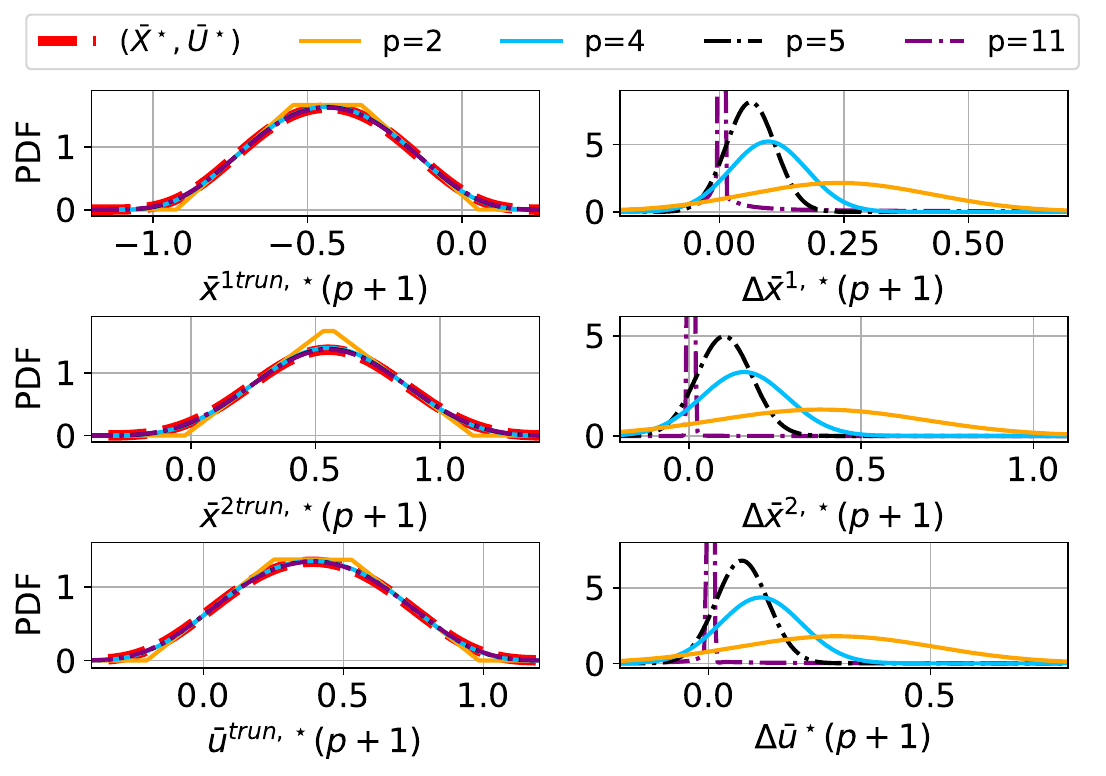}
		\caption{Approximation of $(\bar{X}^\star,\bar{U}^\star)$ via truncated PCEs for different $p$. 
			Left: PDFs of $(\bar{X}^\star,\bar{U}^\star)$ and of $\big(\bar{X}^{\text{trun},\star},\bar{U}^{\text{trun},\star}\big)$; right: PDFs of $\big(\Delta\bar{X}^{\star},\Delta\bar{U}^{\star}\big)$.\label{fig:CSTRDist}} 		
	\end{center}
\end{figure}

Analytically, the closed-form expression~\eqref{eq:InfXU} gives
\begin{align*}
	\mu_X^\sinf &\ed [ -0.437,~0.554 ]^\top + \textstyle{\sum_{k=0}^{\infty}} \tilde{A}^{k}[1,~1]^\top\big(W_k-\mean[W_k]\big),\\
	\mu_U^\sinf &\ed 0.390 + K\textstyle{\sum_{k=0}^{\infty}} \tilde{A}^{k}[1,~1]^\top\big(W_k-\mean[W_k]\big)
\end{align*}
with $ K {=} [ 1.25,~-0.0344 ]$ and $\tilde{A}{=}A{+}BK{=} \big[\begin{smallmatrix} 0.614 & 0.0172 \\ 0.746 & 0.183\phantom{0} \end{smallmatrix}\big]$. While the expectation of $(\mu_X^\sinf,\mu_U^\sinf)$ is straightforward, we calculate its covariance from Lemma~\ref{lem:Limit} as $\Sigma[\mu_X^\sinf] {=}\big[\begin{smallmatrix} 0.0502 & 0.0608 \\ 0.0608 & 0.0772\phantom{0} \end{smallmatrix}\big]$ and $\Sigma[\mu_U^\sinf] {=} 0.0736$.

We choose the optimization horizon of stochastic OCP~\eqref{eq:StochOCP} to be $N=30$ and implement the PCE reformulated OCP~\eqref{eq:StochPCE} in \texttt{julia} using \texttt{JuMP.jl} \citep{dunning17jump} and \texttt{PolyChaos.jl} \citep{muehlpfordt20uncertainty}. We directly solve the numerical optimization problem and obtain the solution in PCE coefficients $\{(\pce{x}_k^{j,\star},\pce{u}_k^{j,\star})\}_{k=0}^{29}$, $j\in\I_{[-2,29]}$. Then we compare the results with the analytical solution in PCE coefficients, i.e. \eqref{eq:SolutionPCE} given by Lemma~\ref{lem:SolutionPCEj} and the maximum difference is $5\cdot10^{-16}$. We depict $\{(\pce{x}_k^{j,\star},\pce{u}_k^{j,\star})\}_{k=0}^{29}$ for all $j\in\I_{[-2,29]}$ in Figure~\ref{fig:CSTRPCEs}, where $\pce{x}^{ij,\star}$ denotes the $j$-th PCE coefficient of the $i$-th component of $X^\star$.  We observe that the computed trajectories $\{(\pce{x}_k^{j,\star},\pce{u}_k^{j,\star})\}_{k=0}^{29}$, $j\in\I_{[-2,29]}$ are in line with the analytical results that are illustrated in Figure~\ref{fig:PCEIllustration}. Note that there is a leaving arc for the trajectory $\{(\pce{x}_k^{-2,\star},\pce{u}_k^{2,\star})\}_{k=0}^{29}$ as the system in~\eqref{eq:StochPCEExp} for $j=-2$ includes the constant $E\mean[W]$.

Then we verify the convergence of the infinite-horizon optimal trajectory shown in Lemma~\ref{lem:Limit} with the considered example. 
As the infinite-horizon OCP~\eqref{eq:StochOCPInf} is in general difficult to be solved numerically, we solve the finite-horizon OCP~\eqref{eq:StochOCP} for a long horizon $N=60$. Then we choose the state-input pair $(X_{30}^\star,U_{30}^\star)$ in the middle of the optimal trajectory to mimic $(X_\infty^\sinf,U_\infty^\sinf)$. We also need to compute the limit $\lim_{k\to\infty} \big(\mu_{X_k^{\sinf}},\mu_{U_k^{\sinf}}\big)$, i.e. $(\mu_X^\sinf,\mu_U^\sinf)$, as given in \eqref{eq:InfXU} in the PCE framework. Since $(\mu_X^\sinf,\mu_U^\sinf)$ contains infinitely many terms, we truncate them after the term containing $\tilde{A}^{99}$, as the largest entry of $|\tilde{A}^{99}|$ is $1.28\cdot 10^{-19}$. To compare the probability of $(X_{30}^\star,U_{30}^\star)$ with $(\mu_X^\sinf,\mu_U^\sinf)$,  we employ the characteristic function, which is the Fourier transform of Probability Density Functions (PDF), and its inverse to approximate the PDFs of $(X_{30}^\star,U_{30}^\star)$ and $(\mu_X^\sinf,\mu_U^\sinf)$. The maximum difference between their PDFs is only $1.91\cdot 10^{-5}$. This is consistent with Lemma~\ref{lem:Limit}.

Next we compare the optimal stationary pair $(\bar{X}^\star,\bar{U}^\star)$, which corresponds to the probability measure $(\mu_X^\sinf,\mu_U^\sinf)$, to its approximation via PCE. Here we consider the error bounds $\delta\in\{0.1,0.01\}$. Then via Lemma~\ref{lem:FiniteApprox} we get the corresponding PCE dimensions of the approximation as $\tilde{p}(0.1)=5$ and $\tilde{p}(0.01)=11$. In comparison to $\tilde{p}$, we also calculate the dimension $\bar{p}$ via \eqref{eq:LeastDimensionAlternative}. We obtain $\bar{p}(0.1)=2$ and $\bar{p}(0.01)=4$. We see that $\bar{p}(\delta)<\tilde{p}(\delta)$ holds for both $\delta=0.1$ and $\delta=0.01$. Then we compute all the PDFs of the approximations $\big(\bar{X}^{\text{trun},\star}(p+1),\bar{U}^{\text{trun},\star}(p+1)\big)$ for different $p$ as~\eqref{eq:ApproximationPCE} and the PDF of $(\bar{X}^\star,\bar{U}^\star)$, which are illustrated in the left subplots of Figure~\ref{fig:CSTRDist}. Additionally, we plot the PDFs of the truncation errors $\big(\Delta \bar{X}^{,\star}(p+1),\Delta \bar{U}^{,\star}(p+1)\big)$ for different $p$ in the right subplots of Figure~\ref{fig:CSTRDist}.
Note that the approximation accuracy increases as the PDF of $(\Delta \bar{X}^{\star},\Delta \bar{U}^{\star})$ converges to the Dirac distribution $\delta(0)$.
Therefore, we observe that a higher-dimensional PCE offers a more accurate approximation.
Moreover, the PDF of the truncation error for $p=11$ is almost a Dirac distribution, while the impulse at 0 is not fully presented in Figure~\ref{fig:CSTRDist} due to the space limit on the $y$-axis. Specifically, the truncation error for $p=11$ is less than $1.64\cdot 10^{-5}$ in the sense of the Wasserstein metric.
\section{Conclusions} \label{sec:Conclusion}
This paper has addressed stochastic LQR problems for discrete-time LTI systems for non-Gaussian disturbances with finite expectation and variance. In contrast to the established moment-based approach, the proposed PCE scheme allows uncertainty propagation, e.g. distribution propagation over the horizon. The crucial insight of our work is that all sources of uncertainties, i.e. the uncertain initial condition and the process disturbances at each time step, can be decoupled from each other and thus handled individually. This decoupling allows for a structure exploiting moving horizon basis truncation for which we have given error bounds. Moreover, we have analyzed the convergence properties of the optimal state and input trajectories for the infinite-horizon case.

We have also characterized the stochastic stationary optimization problem and given its unique solution, i.e. the optimal stationary pair, in closed analytic form. We have shown that the optimal stationary pair is indeed the limit of the optimal trajectory of the corresponding infinite-horizon LQR problem and is thus of infinite dimension. Importantly, for an arbitrary desired error bound of the approximation error, we have provided finite-dimensional approximations of the optimal stationary pair.

\section*{Acknowledgements}
The authors acknowledge funding by the Deutsche Forschungsgemeinschaft (DFG, German Research Foundation) - project number 499435839.
                         
\bibliography{arXivRef}

\begin{thebibliography}{42}
\expandafter\ifx\csname natexlab\endcsname\relax\def\natexlab#1{#1}\fi
\providecommand{\url}[1]{\texttt{#1}}
\providecommand{\href}[2]{#2}
\providecommand{\path}[1]{#1}
\providecommand{\DOIprefix}{doi:}
\providecommand{\ArXivprefix}{arXiv:}
\providecommand{\URLprefix}{URL: }
\providecommand{\Pubmedprefix}{pmid:}
\providecommand{\doi}[1]{\href{http://dx.doi.org/#1}{\path{#1}}}
\providecommand{\Pubmed}[1]{\href{pmid:#1}{\path{#1}}}
\providecommand{\bibinfo}[2]{#2}
\ifx\xfnm\relax \def\xfnm[#1]{\unskip,\space#1}\fi
\bibitem[{Ahbe et~al.(2020)Ahbe, Iannelli \& Smith}]{ahbe20region}
\bibinfo{author}{Ahbe, E.}, \bibinfo{author}{Iannelli, A.}, \&
  \bibinfo{author}{Smith, R.} (\bibinfo{year}{2020}).
\newblock \bibinfo{title}{Region of attraction analysis of nonlinear stochastic
  systems using polynomial chaos expansion}.
\newblock {\it \bibinfo{journal}{Automatica}\/},  {\it
  \bibinfo{volume}{122}\/}, \bibinfo{pages}{109187}.
\bibitem[{Anderson \& Moore(1979)}]{anderson79optimal}
\bibinfo{author}{Anderson, B.}, \& \bibinfo{author}{Moore, J.}
  (\bibinfo{year}{1979}).
\newblock {\it \bibinfo{title}{Optimal Filtering}\/}.
\newblock \bibinfo{publisher}{Prentice-Hall}.
\bibitem[{Anderson \& Moore(1989)}]{anderson89optimal}
\bibinfo{author}{Anderson, B.}, \& \bibinfo{author}{Moore, J.}
  (\bibinfo{year}{1989}).
\newblock {\it \bibinfo{title}{{Optimal Control: Linear Quadratic Methods}}\/}.
\newblock \bibinfo{publisher}{Prentice-Hall}.
\bibitem[{{\AA}str{\"o}m(1970)}]{astrom70introduction}
\bibinfo{author}{{\AA}str{\"o}m, K.} (\bibinfo{year}{1970}).
\newblock {\it \bibinfo{title}{Introduction to Stochastic Control Theory}\/}.
\newblock \bibinfo{publisher}{Academic Press}.
\bibitem[{Bertsekas \& Tsitsiklis(2008)}]{bertsekas08introduction}
\bibinfo{author}{Bertsekas, D.}, \& \bibinfo{author}{Tsitsiklis, J.}
  (\bibinfo{year}{2008}).
\newblock {\it \bibinfo{title}{Introduction to Probability}\/}
  volume~\bibinfo{volume}{1}.
\newblock \bibinfo{publisher}{Athena Scientific}.
\bibitem[{Cameron \& Martin(1947)}]{cameron47orthogonal}
\bibinfo{author}{Cameron, R.}, \& \bibinfo{author}{Martin, W.}
  (\bibinfo{year}{1947}).
\newblock \bibinfo{title}{The orthogonal development of non-linear functionals
  in series of {F}ourier-{H}ermite functionals}.
\newblock {\it \bibinfo{journal}{Annals of Mathematics}\/},  (pp.
  \bibinfo{pages}{385--392}).
\bibitem[{Carlson et~al.(1991)Carlson, Haurie \&
  Leizarowitz}]{carlson91infinite}
\bibinfo{author}{Carlson, D.}, \bibinfo{author}{Haurie, A.}, \&
  \bibinfo{author}{Leizarowitz, A.} (\bibinfo{year}{1991}).
\newblock {\it \bibinfo{title}{{Infinite Horizon Optimal Control: Deterministic
  and Stochastic Systems}}\/}.
\newblock \bibinfo{publisher}{Springer-Verlag}.
\bibitem[{Dunning et~al.(2017)Dunning, Huchette \& Lubin}]{dunning17jump}
\bibinfo{author}{Dunning, I.}, \bibinfo{author}{Huchette, J.}, \&
  \bibinfo{author}{Lubin, M.} (\bibinfo{year}{2017}).
\newblock \bibinfo{title}{{JuMP}: A modeling language for mathematical
  optimization}.
\newblock {\it \bibinfo{journal}{SIAM Review}\/},  {\it
  \bibinfo{volume}{59}\/}(2), \bibinfo{pages}{295--320}.
\bibitem[{Ernst et~al.(2012)Ernst, Mugler, Starkloff \&
  Ullmann}]{ernst12convergence}
\bibinfo{author}{Ernst, O.}, \bibinfo{author}{Mugler, A.},
  \bibinfo{author}{Starkloff, H.-J.}, \& \bibinfo{author}{Ullmann, E.}
  (\bibinfo{year}{2012}).
\newblock \bibinfo{title}{On the convergence of generalized polynomial chaos
  expansions}.
\newblock {\it \bibinfo{journal}{ESAIM: Mathematical Modelling and Numerical
  Analysis}\/},  {\it \bibinfo{volume}{46}\/}(2), \bibinfo{pages}{317--339}.
\bibitem[{Fagiano \& Khammash(2012)}]{fagiano12nonlinear}
\bibinfo{author}{Fagiano, L.}, \& \bibinfo{author}{Khammash, M.}
  (\bibinfo{year}{2012}).
\newblock \bibinfo{title}{Nonlinear stochastic model predictive control via
  regularized polynomial chaos expansions}.
\newblock In {\it \bibinfo{booktitle}{51st IEEE Conference on Decision and
  Control (CDC)}\/} (pp. \bibinfo{pages}{142--147}).
\newblock \bibinfo{organization}{IEEE}.
\bibitem[{Faulwasser \& Gr\"une(2022)}]{faulwasser22turnpike}
\bibinfo{author}{Faulwasser, T.}, \& \bibinfo{author}{Gr\"une, L.}
  (\bibinfo{year}{2022}).
\newblock \bibinfo{title}{Turnpike properties in optimal control: An overview
  of discrete-time and continuous-time results}.
\newblock In \bibinfo{editor}{E.~Zuazua}, \& \bibinfo{editor}{E.~Trelat}
  (Eds.), {\it \bibinfo{booktitle}{Handbook of Numerical Analysis}\/}
  chapter~\bibinfo{chapter}{11}. (pp. \bibinfo{pages}{367--400}).
\newblock \bibinfo{publisher}{Elsevier} volume~\bibinfo{volume}{23}.
\bibitem[{Faulwasser et~al.(2023)Faulwasser, Ou, Pan, Schmitz \&
  Worthmann}]{faulwasser23behavioral}
\bibinfo{author}{Faulwasser, T.}, \bibinfo{author}{Ou, R.},
  \bibinfo{author}{Pan, G.}, \bibinfo{author}{Schmitz, P.}, \&
  \bibinfo{author}{Worthmann, K.} (\bibinfo{year}{2023}).
\newblock \bibinfo{title}{Behavioral theory for stochastic systems? {A}
  data-driven journey from willems to wiener and back again}.
\newblock {\it \bibinfo{journal}{Annual Reviews in Control}\/},  {\it
  \bibinfo{volume}{55}\/}, \bibinfo{pages}{92--117}.
\bibitem[{Faulwasser \& Zanon(2018)}]{faulwasser18asymptotic}
\bibinfo{author}{Faulwasser, T.}, \& \bibinfo{author}{Zanon, M.}
  (\bibinfo{year}{2018}).
\newblock \bibinfo{title}{Asymptotic stability of economic {NMPC}: The
  importance of adjoints}.
\newblock {\it \bibinfo{journal}{IFAC-PapersOnLine}\/},  {\it
  \bibinfo{volume}{51}\/}(20), \bibinfo{pages}{157--168}.
\bibitem[{Feller(1971)}]{feller71introduction}
\bibinfo{author}{Feller, W.} (\bibinfo{year}{1971}).
\newblock {\it \bibinfo{title}{An Introduction to Probability Theory and Its
  Applications}\/}.
\newblock \bibinfo{publisher}{John Wiley \& Sons, Inc}.
\bibitem[{Fisher \& Bhattacharya(2009)}]{fisher09linear}
\bibinfo{author}{Fisher, J.}, \& \bibinfo{author}{Bhattacharya, R.}
  (\bibinfo{year}{2009}).
\newblock \bibinfo{title}{Linear quadratic regulation of systems with
  stochastic parameter uncertainties}.
\newblock {\it \bibinfo{journal}{Automatica}\/},  {\it
  \bibinfo{volume}{45}\/}(12), \bibinfo{pages}{2831--2841}.
\bibitem[{Fristedt \& Gray(1997)}]{fristedt13modern}
\bibinfo{author}{Fristedt, B.}, \& \bibinfo{author}{Gray, L.}
  (\bibinfo{year}{1997}).
\newblock {\it \bibinfo{title}{A Modern Approach to Probability Theory}\/}.
\newblock \bibinfo{publisher}{Birkh{\"a}user Boston}.
\bibitem[{Gattami(2009)}]{gattami09generalized}
\bibinfo{author}{Gattami, A.} (\bibinfo{year}{2009}).
\newblock \bibinfo{title}{Generalized linear quadratic control}.
\newblock {\it \bibinfo{journal}{IEEE Transactions on Automatic Control}\/},
  {\it \bibinfo{volume}{55}\/}(1), \bibinfo{pages}{131--136}.
\bibitem[{Ghanem \& Spanos(1991)}]{ghanem91stochastic}
\bibinfo{author}{Ghanem, R.~G.}, \& \bibinfo{author}{Spanos, P.~D.}
  (\bibinfo{year}{1991}).
\newblock {\it \bibinfo{title}{Stochastic Finite Elements: A Spectral
  Approach}\/}.
\newblock \bibinfo{publisher}{Springer-Verlag}.
\bibitem[{Kallenberg(1997)}]{kallenberg97foundations}
\bibinfo{author}{Kallenberg, O.} (\bibinfo{year}{1997}).
\newblock {\it \bibinfo{title}{Foundations of Modern Probability}\/}.
\newblock \bibinfo{publisher}{Springer}.
\bibitem[{Kalman(1960{\natexlab{a}})}]{kalman60contributions}
\bibinfo{author}{Kalman, R.} (\bibinfo{year}{1960}{\natexlab{a}}).
\newblock \bibinfo{title}{Contributions to the theory of optimal control}.
\newblock {\it \bibinfo{journal}{Bolet\'{i}n de la Sociedad Matemática
  Mexicana}\/},  {\it \bibinfo{volume}{5}\/}(2), \bibinfo{pages}{102--119}.
\bibitem[{Kalman(1960{\natexlab{b}})}]{kalman60new}
\bibinfo{author}{Kalman, R.} (\bibinfo{year}{1960}{\natexlab{b}}).
\newblock \bibinfo{title}{A new approach to linear filtering and prediction
  problems}.
\newblock {\it \bibinfo{journal}{Journal of Basic Engineering}\/},  {\it
  \bibinfo{volume}{82}\/}(1), \bibinfo{pages}{35--45}.
\bibitem[{Kim \& Braatz(2012)}]{kim12probabilistic}
\bibinfo{author}{Kim, K.}, \& \bibinfo{author}{Braatz, R.}
  (\bibinfo{year}{2012}).
\newblock \bibinfo{title}{Probabilistic analysis and control of uncertain
  dynamic systems: Generalized polynomial chaos expansion approaches}.
\newblock In {\it \bibinfo{booktitle}{2012 American Control Conference}\/} (pp.
  \bibinfo{pages}{44--49}).
\bibitem[{Kim et~al.(2013)Kim, Shen, Nagy \& Braatz}]{kim13wiener}
\bibinfo{author}{Kim, K.~K.}, \bibinfo{author}{Shen, D.~E.},
  \bibinfo{author}{Nagy, Z.~K.}, \& \bibinfo{author}{Braatz, R.~D.}
  (\bibinfo{year}{2013}).
\newblock \bibinfo{title}{Wiener's polynomial chaos for the analysis and
  control of nonlinear dynamical systems with probabilistic uncertainties
  [{H}istorical {P}erspectives]}.
\newblock {\it \bibinfo{journal}{IEEE Control Systems Magazine}\/},  {\it
  \bibinfo{volume}{33}\/}(5), \bibinfo{pages}{58--67}.
\bibitem[{Lefebvre(2020)}]{lefebvre20moment}
\bibinfo{author}{Lefebvre, T.} (\bibinfo{year}{2020}).
\newblock \bibinfo{title}{On moment estimation from polynomial chaos expansion
  models}.
\newblock {\it \bibinfo{journal}{IEEE Control Systems Letters}\/},  {\it
  \bibinfo{volume}{5}\/}(5), \bibinfo{pages}{1519--1524}.
\bibitem[{Levajkovi{\'c} et~al.(2018)Levajkovi{\'c}, Mena \&
  Pfurtscheller}]{levajkovic18solving}
\bibinfo{author}{Levajkovi{\'c}, T.}, \bibinfo{author}{Mena, H.}, \&
  \bibinfo{author}{Pfurtscheller, L.-M.} (\bibinfo{year}{2018}).
\newblock \bibinfo{title}{Solving stochastic {LQR} problems by polynomial
  chaos}.
\newblock {\it \bibinfo{journal}{IEEE Control Systems Letters}\/},  {\it
  \bibinfo{volume}{2}\/}(4), \bibinfo{pages}{641--646}.
\bibitem[{Lim \& Zhou(1999)}]{lim99stochastic}
\bibinfo{author}{Lim, A. E.~B.}, \& \bibinfo{author}{Zhou, X.~Y.}
  (\bibinfo{year}{1999}).
\newblock \bibinfo{title}{Stochastic optimal {LQR} control with integral
  quadratic constraints and indefinite control weights}.
\newblock {\it \bibinfo{journal}{IEEE Transactions on Automatic Control}\/},
  {\it \bibinfo{volume}{44}\/}(7), \bibinfo{pages}{1359--1369}.
\bibitem[{M{\"u}hlpfordt(2020)}]{muehlpfordt20uncertainty}
\bibinfo{author}{M{\"u}hlpfordt, T.} (\bibinfo{year}{2020}).
\newblock {\it \bibinfo{title}{Uncertainty quantification via polynomial chaos
  expansion -- {M}ethods and applications for optimization of power
  systems}\/}.
\newblock Ph.D. thesis Karlsruher Institut f{\"u}r Technologie (KIT).
\bibitem[{Pan et~al.(2023)Pan, Ou \& Faulwasser}]{pan23stochastic}
\bibinfo{author}{Pan, G.}, \bibinfo{author}{Ou, R.}, \&
  \bibinfo{author}{Faulwasser, T.} (\bibinfo{year}{2023}).
\newblock \bibinfo{title}{On a stochastic fundamental lemma and its use for
  data-driven optimal control}.
\newblock {\it \bibinfo{journal}{IEEE Transactions on Automatic Control}\/},
  {\it \bibinfo{volume}{68}\/}(10), \bibinfo{pages}{5922--5937}.
\bibitem[{Paulson et~al.(2014)Paulson, Mesbah, Streif, Findeisen \&
  Braatz}]{paulson14fast}
\bibinfo{author}{Paulson, J.}, \bibinfo{author}{Mesbah, A.},
  \bibinfo{author}{Streif, S.}, \bibinfo{author}{Findeisen, R.}, \&
  \bibinfo{author}{Braatz, R.} (\bibinfo{year}{2014}).
\newblock \bibinfo{title}{Fast stochastic model predictive control of
  high-dimensional systems}.
\newblock In {\it \bibinfo{booktitle}{53rd IEEE Conference on Decision and
  Control (CDC)}\/} (pp. \bibinfo{pages}{2802--2809}).
\newblock \bibinfo{organization}{IEEE}.
\bibitem[{Paulson et~al.(2015)Paulson, Streif \& Mesbah}]{paulson15stability}
\bibinfo{author}{Paulson, J.~A.}, \bibinfo{author}{Streif, S.}, \&
  \bibinfo{author}{Mesbah, A.} (\bibinfo{year}{2015}).
\newblock \bibinfo{title}{Stability for receding-horizon stochastic model
  predictive control}.
\newblock In {\it \bibinfo{booktitle}{2015 American Control Conference}\/} (pp.
  \bibinfo{pages}{937--943}).
\newblock \bibinfo{organization}{IEEE}.
\bibitem[{R{\"u}schendorf(1985)}]{ruschendorf85wasserstein}
\bibinfo{author}{R{\"u}schendorf, L.} (\bibinfo{year}{1985}).
\newblock \bibinfo{title}{The {W}asserstein distance and approximation
  theorems}.
\newblock {\it \bibinfo{journal}{Probability Theory and Related Fields}\/},
  {\it \bibinfo{volume}{70}\/}(1), \bibinfo{pages}{117--129}.
\bibitem[{Schie{\ss}l et~al.(2023)Schie{\ss}l, Ou, Faulwasser, Baumann \&
  Gr{\"u}ne}]{schiessl23pathwise}
\bibinfo{author}{Schie{\ss}l, J.}, \bibinfo{author}{Ou, R.},
  \bibinfo{author}{Faulwasser, T.}, \bibinfo{author}{Baumann, M.}, \&
  \bibinfo{author}{Gr{\"u}ne, L.} (\bibinfo{year}{2023}).
\newblock \bibinfo{title}{Pathwise turnpike and dissipativity results for
  discrete-time stochastic linear-quadratic optimal control problems}.
\newblock In {\it \bibinfo{booktitle}{62nd IEEE Conference on Decision and
  Control (CDC)}\/} (pp. \bibinfo{pages}{2790--2795}).
\newblock \bibinfo{organization}{IEEE}.
\bibitem[{Schie{\ss}l et~al.(2024)Schie{\ss}l, Ou, Faulwasser, Baumann \&
  Gr{\"u}ne}]{schiessl24turnpike}
\bibinfo{author}{Schie{\ss}l, J.}, \bibinfo{author}{Ou, R.},
  \bibinfo{author}{Faulwasser, T.}, \bibinfo{author}{Baumann, M.}, \&
  \bibinfo{author}{Gr{\"u}ne, L.} (\bibinfo{year}{2024}).
\newblock \bibinfo{title}{Turnpike and dissipativity in generalized
  discrete-time stochastic linear-quadratic optimal control}.
\newblock {\it \bibinfo{journal}{SIAM Journal on Control and Optimization,
  accepted}\/}, .
\bibitem[{Simoncini(2016)}]{simoncini16computational}
\bibinfo{author}{Simoncini, V.} (\bibinfo{year}{2016}).
\newblock \bibinfo{title}{Computational methods for linear matrix equations}.
\newblock {\it \bibinfo{journal}{SIAM Review}\/},  {\it
  \bibinfo{volume}{58}\/}(3), \bibinfo{pages}{377--441}.
\bibitem[{Singh \& Pal(2017)}]{singh17extended}
\bibinfo{author}{Singh, A.}, \& \bibinfo{author}{Pal, B.}
  (\bibinfo{year}{2017}).
\newblock \bibinfo{title}{An extended linear quadratic regulator for {LTI}
  systems with exogenous inputs}.
\newblock {\it \bibinfo{journal}{Automatica}\/},  {\it \bibinfo{volume}{76}\/},
  \bibinfo{pages}{10--16}.
\bibitem[{Sullivan(2015)}]{sullivan15introduction}
\bibinfo{author}{Sullivan, T.} (\bibinfo{year}{2015}).
\newblock {\it \bibinfo{title}{{Introduction to Uncertainty Quantification}}\/}
  volume~\bibinfo{volume}{63}.
\newblock \bibinfo{publisher}{Springer International}.
\bibitem[{Sun \& Yong(2018)}]{sun18stochastic}
\bibinfo{author}{Sun, J.}, \& \bibinfo{author}{Yong, J.}
  (\bibinfo{year}{2018}).
\newblock \bibinfo{title}{Stochastic linear quadratic optimal control problems
  in infinite horizon}.
\newblock {\it \bibinfo{journal}{Applied Mathematics \& Optimization}\/},  {\it
  \bibinfo{volume}{78}\/}(1), \bibinfo{pages}{145--183}.
\bibitem[{Templeton et~al.(2012)Templeton, Ahmadian \&
  Southward}]{templeton12probabilistic}
\bibinfo{author}{Templeton, B.}, \bibinfo{author}{Ahmadian, M.}, \&
  \bibinfo{author}{Southward, S.} (\bibinfo{year}{2012}).
\newblock \bibinfo{title}{Probabilistic control using {H}2 control design and
  polynomial chaos: Experimental design, analysis, and results}.
\newblock {\it \bibinfo{journal}{Probabilistic Engineering Mechanics}\/},  {\it
  \bibinfo{volume}{30}\/}, \bibinfo{pages}{9--19}.
\bibitem[{Villani(2009)}]{villani09wasserstein}
\bibinfo{author}{Villani, C.} (\bibinfo{year}{2009}).
\newblock \bibinfo{title}{The {W}asserstein distances}.
\newblock {\it \bibinfo{journal}{Optimal Transport: Old and New}\/},  (pp.
  \bibinfo{pages}{93--111}).
\bibitem[{Wan et~al.(2023)Wan, Shen, Lucia, Findeisen \&
  Braatz}]{wan23polynomial}
\bibinfo{author}{Wan, Y.}, \bibinfo{author}{Shen, D.}, \bibinfo{author}{Lucia,
  S.}, \bibinfo{author}{Findeisen, R.}, \& \bibinfo{author}{Braatz, R.}
  (\bibinfo{year}{2023}).
\newblock \bibinfo{title}{A polynomial chaos approach to robust static
  output-feedback control with bounded truncation error}.
\newblock {\it \bibinfo{journal}{IEEE Transactions on Automatic Control}\/},
  {\it \bibinfo{volume}{68}\/}(1), \bibinfo{pages}{470--477}.
\bibitem[{Wiener(1938)}]{wiener38homogeneous}
\bibinfo{author}{Wiener, N.} (\bibinfo{year}{1938}).
\newblock \bibinfo{title}{The homogeneous chaos}.
\newblock {\it \bibinfo{journal}{American Journal of Mathematics}\/},  (pp.
  \bibinfo{pages}{897--936}).
\bibitem[{Xiu \& Karniadakis(2002)}]{xiu02wiener}
\bibinfo{author}{Xiu, D.}, \& \bibinfo{author}{Karniadakis, G.}
  (\bibinfo{year}{2002}).
\newblock \bibinfo{title}{The {W}iener--{A}skey polynomial chaos for stochastic
  differential equations}.
\newblock {\it \bibinfo{journal}{SIAM Journal on Scientific Computing}\/},
  {\it \bibinfo{volume}{24}\/}(2), \bibinfo{pages}{619--644}.

\end{thebibliography}
\end{document}